\documentclass{amsart}

\usepackage{amsmath}
\usepackage{amsthm}
\usepackage{amsfonts}
\usepackage{amssymb}

\newcommand\R{{\mathbf{R}}}
\newcommand\C{{\mathbf{C}}}

\newcommand\Z{{\mathbf{Z}}}

\newcommand\eps{\varepsilon}

\newcommand\Lip{\operatorname{Lip}}

\newcommand\dist{\operatorname{dist}}

\newcommand\id{{\operatorname{id}}}
\newcommand\sgn{\operatorname{sgn}}

\newcommand\Vol{\operatorname{Vol}}
\newcommand\bigO{\mathcal{O}}

\theoremstyle{plain}
  \newtheorem{theorem}[subsection]{Theorem}

  \newtheorem{proposition}[subsection]{Proposition}
  
  \newtheorem{lemma}[subsection]{Lemma}
  \newtheorem{corollary}[subsection]{Corollary}

\theoremstyle{remark}
  \newtheorem{remark}[subsection]{Remark}
  \newtheorem{example}[subsection]{Example}

\theoremstyle{definition}
  \newtheorem{definition}[subsection]{Definition}

\include{psfig}

\begin{document}

\title{A finitary version of Gromov's polynomial growth theorem}

\author{Yehuda Shalom}
\address{Department of Mathematics, UCLA, Los Angeles CA,  90095-1555}
\email{yeshalom@math.ucla.edu}

\author{Terence Tao}
\address{Department of Mathematics, UCLA, Los Angeles CA 90095-1555}
\email{tao@math.ucla.edu}

\begin{abstract} We show that for some absolute (explicit) constant
$C$, the following holds for every finitely
generated group $G$, and all $d >0$: \newline 
If there is {\it some} $ R_0 > \exp(\exp(Cd^C))$ 
for which the number of elements in a ball of radius $R_0$ in a Cayley graph 
of $G$ is bounded by
$R_0^d$, then $G$ has a finite index subgroup which is nilpotent 
(of step $<C^d$). An effective bound on the finite index is provided if
``nilpotent'' is replaced by ``polycyclic'', thus yielding
a non-trivial result for finite groups as well. 
\end{abstract}

\maketitle

\section{Introduction}

A famous theorem of Gromov \cite{gromov} asserts that all finitely 
generated groups of polynomial growth (thus, in the notation used below, one has $|B_S(R)| \leq R^d$ for some $d$ and all sufficiently large $R$, where $S$ is a fixed set of generators) are virtually nilpotent.  This was generalized by van der Dries and Wilkie \cite{vdw} by assuming the polynomial growth condition $|B_S(R)| \leq R^d$ at an infinite number, rather than all scales.  A second proof of this fact, which for the first time released the
dependence on the involved solution to Hilbert's $5^{\operatorname{th}}$ problem, 
was given recently by Kleiner \cite{kleiner}, motivated by work of 
Colding-Minicozzi~\cite{cold} and 
using the theory of harmonic functions. 
In this paper, we refine Kleiner's work to 
obtain a further strengthening of Gromov's theorem that only requires the polynomial growth condition at \emph{one} (sufficiently large, yet explicit) 
scale, and gives quantitative control on the nilpotency degree.  
To state precisely the result (mentioned at the abstract above), we need some notation.

\begin{definition}[$(R_0,d)$-growth groups]\label{rdg-def}  A \emph{finitely generated group} is a pair $G = (G,S)$, where $G$ is a group, and $S \subset G$ is a finite non-empty symmetric set which generates $G$ (thus $S^{-1} := \{ s^{-1}: s \in S \} = S$).  For each $g \in G$, we define $\|g\|_S := \inf \{ n: g \in S^n \}$ to be the length of the smallest word with alphabet $S$ that evaluates to $g$. For each $R > 0$, we define $B_S(R) := \{ g \in G: \|g\|_S \leq R \}$ to be the collection of words in $S$ of length at most $R$.  For $R_0, d > 0$, we define a \emph{$(R_0,d)$-growth group} to be a finitely generated group $(G,S)$ such that
\begin{equation}\label{bsrd}
 |B_S(R_0)| \leq R_0^d.
\end{equation}
\end{definition}

\begin{remark} One can of course generalize \eqref{bsrd} by replacing the right-hand side of $R_0^d$ by $CR_0^d$ for some additional parameter $C$, as is customary in the literature.  But this would add a new parameter to an already complicated notational system and so we have chosen to drop this parameter, as one can partially simulate it by increasing $d$ slightly and assuming $R_0$ is large.  Note that a finitely generated group $(G,S)$ has polynomial growth if and only if there exists a $d$ such that $(G,S)$ is a $(R_0,d)$-growth group for all sufficiently large $R_0$.
\end{remark}

\begin{definition}[Quantitative finite index]\label{qfi-def}  Let $R, K > 0$.  A finitely generated group $(G',S')$ is said to be a \emph{$(K,R)$-subgroup} of $(G,S)$ if $G'$ is a subgroup of $G$, $S' \subset B_S(R)$, and $B_S(K+1) \subset B_S(K) \cdot B_{S'}(K)$.  If we drop the final condition $B_S(K+1) \subset B_S(K) \cdot B_{S'}(K)$, we say that $G'$ is a \emph{$(\infty,R)$-subgroup} of $G$.
\end{definition}

\begin{remark}\label{joe} Observe that if $(G',S')$ is a $(K,R)$-subgroup of $(G,S)$ then $B_S(r+1) \subset B_S(r) \cdot B_{S'}(K)$ for all $r \geq K$, and on iterating this we conclude $G \subset B_S(K) \cdot G'$.  In particular, $G'$ is a finite index subgroup of $G$ of index at most $|B_S(K)|$.   Conversely, if $(G',S')$ has finite index in $(G,S)$, then we can write $G$ as a finite union of cosets $x_1 \cdot G', \ldots,x_m \cdot G'$ of $G'$, and in particular one has the relations $e x_i = x_{j_{e,i}} g_{e,i}$ for all $1 \leq i \leq m$, $e \in S$, and some $1 \leq j_{e,i} \in m$, $g_{e,i} \in G'$.  If one then sets $R := \sup \{ \|e'\|_S: e' \in S' \}$ and $K := \sup \{ \|x_i\|_S: 1 \leq i \leq m\} \cup \{ \|g_{e,i}\|_{S'}: 1 \leq i \leq m, e \in S \}$ we see that $(G',S')$ is a $(K,R)$-subgroup.  Note, however, that this argument does not (and cannot) 
give effective bounds on $K, R$.
\end{remark}

\begin{example} Let $m=2p-1$ be a large odd number. 
Then the (additive) finitely generated group $(\Z/m\Z, \{-2,+2\})$ 
is a $(2,2)$-subgroup of 
$(\Z/m\Z, \{-1, +1\})$, while conversely, $(\Z/m\Z, \{-1, +1\})$ is  
merely a $(2,p)$-subgroup  
of $(\Z/m\Z, \{-2,+2\})$.  The intuition here is that while $\{-2,+2\}$ does generate the element $1$, this is a ``global'' fact (relying on the parity of $m$) rather than a ``local'' one, and thus cannot be detected in the limit $p\to \infty$ if one is only allowed to perform a bounded number of group operations.  Thus we see that this quantitative notion of finite index not only measures the index of $G'$ in $G$, but also the relative position of $S$ and $S'$. This kind of ``practical'' interpretation of abstract (and often trivial) group theoretic notions is necessary in order to perform the quantitative arguments in this paper properly.
\end{example}

\begin{remark}\label{growth} It is immediate that if $(G,S)$ is a $(R_0,d)$-growth group, and $(G',S')$ is a $(K,R_0^\kappa)$-subgroup of $(G,S)$ for some $0 < \kappa < 1$, then $(G',S')$ is a $(R_0^{1-\kappa}, d/(1-\kappa))$-growth group.  This is analogous to the obvious fact that any finite index subgroup of a group of polynomial growth, remains of polynomial growth.
\end{remark}

\begin{definition}[Virtual nilpotency]\label{virtnil-def}  Let $K,R, s, D \geq 1$.  A finitely generated group $(G,S)$ is said to be \emph{$(s,D)$-nilpotent} if $|S| \leq D$ and $G$ is nilpotent of step at most $s$ (i.e. every $s'$-fold iterated commutator vanishes for $s'>s$).  A group $(G,S)$ is said to be \emph{$(K,R,s,D)$-virtually nilpotent} if it contains a $(s,D)$-nilpotent $(K,R)$-subgroup $(G',S')$.
\end{definition}

\begin{theorem}[Quantitative Gromov theorem]\label{main-thm}  Let $d,R_0 > 0$,  and assume that
$$
R_0 \geq \exp(\exp(Cd^C))
$$
for some sufficiently large absolute constant $C$.  Then every $(R_0,d)$-growth group is $(K(R_0,d), K(R_0,d), C^d, C^d)$-virtually nilpotent for some $K(R_0,d)$ depending only on $R_0,d$.
\end{theorem}

In order to simplify the exposition somewhat, we do not give an effective bound for $K(R_0,d)$ in our arguments, relying instead on an ineffective compactness argument to establish its finiteness.  However, it is possible to eradicate this compactness argument from the proof by standard ``quantifier elimination'' techniques, at the cost of making it substantially lengthier (and the final bound for $K(R_0,d)$ obtained is quite poor, of Ackermann type in $d$).  On the other hand, the arguments do give an effective value for $C$ (it appears that
$C=100$ works, even if for clarity we won't keep track of this aspect, in which 
no tightness is claimed). 
We discuss full effectivization issues in Section~\ref{effective}, along
with the following result:

\begin{theorem}[Fully quantitative weak Gromov theorem]\label{polycyclic} 
Let $d > 0$ and $R_0 > 0$, and assume that
$$
R_0 \geq \exp(\exp(Cd^C))
$$
for some sufficiently large absolute constant $C$.  Then every 
$(R_0,d)$-growth group has a normal subgroup of index at most 
$\exp(R_0^{\exp(exp(d^C))})$ which is polycyclic. 
\end{theorem}

This result is in fact established along 
Sections \ref{kleinersec-1}-\ref{kleinersec-5} (see Proposition~\ref{trivcor2} below),
and can be recommended to the reader as a natural ``resting point''  
along the way to 
the full proof of the main Theorem~\ref{main-thm}. It captures the 
quantitative 
 outcome made out of 
Kleiner's approach, and avoids completely the Milnor-Wolf  
part of the proof, which is only made semi-quantitative here.
Note also that a completely effective version of Theorem~\ref{main-thm}
is still available when the group is assumed torsion free;
see Corollary~\ref{torsion-free}.

We have the following immediate corollary of Theorem \ref{main-thm}:

\begin{corollary}[Slightly super-polynomial growth implies virtual nilpotency]  Let $(G,S)$ be a finitely generated group such that
$$ |B_S(R)| \leq R^{c (\log \log R)^c}$$
for some $R > 1/c$, where $c>0$ is a sufficiently small absolute constant.  Then $G$ is virtually nilpotent.
\end{corollary}

\begin{remark}
By using completely ineffective compactness arguments, it has been well
known that Gromov's theorem implies the existence of {\it some} super-polynomial growth function for which
the Corollary holds, although no such explicit function was known before. 
It has been proposed that the subradical growth type $\exp(C{\sqrt n})$ might work (see~\cite{Gri},~\cite{Lu-Ma} for some results in this direction). It is also interesting to
compare the situation with Segal's refutation~\cite{Seg} of a conjecture
of Lubotzky, Pyber and Shalev~\cite{LPS} concerning the growth type of
the number of finite index subgroups of a residually finite group (as a function of the index). It has been shown in~\cite{Seg} that here no super-polynomial 
bound could
yield the same characterization as the polynomial one,
of being virtually solvable of finite rank  
(a remarkable result of Lubotzky, Mann and Segal~\cite{LMS}, which relies, among other things, on
Lazard's deep $p$-adic analogue of Hilbert's $5^{\operatorname{th}}$ problem~\cite{Laz}). 
\end{remark}

The nilpotent group generated by Theorem \ref{main-thm} has step at most $C^d$, but an inspection of the proof shows that it also has Hirsch length (i.e., sum
of torsion free ranks of the abelian quotients in a grading)
 at most $C^d$.  It is a standard computation (which also follows from
the well known Bass-Guivarc'h formula for the growth of nilpotent groups), that 
nilpotent groups of 
Hirsch length $r$ and step $s$ have polynomial growth of order at 
most $O( rs )$.  We thus conclude:

\begin{corollary}[Polynomial growth at one scale implies polynomial growth at all scales]\label{allscales}  Let $d > 0$ and $R_0 > 0$, and assume that
$$
R_0 \geq \exp(\exp(Cd^C))
$$
for some sufficiently large absolute constant $C$.  Let $G$ be a $(R_0,d)$-growth group.  Then one has 
$$ |B_S(R)| \leq K R^{C^d}$$
for all $R \geq 1$, where $K = K(R_0,d)$ depends only on $R_0,d$.
\end{corollary}

Thus, polynomial growth of order $d$ at one (moderately large) scale 
implies polynomial 
growth of order $O(1)^d$ at all subsequent scales.  We do not know if this exponential loss in the polynomial growth degree is necessary.

\subsection{Overview of proof}

Our arguments broadly follow the strategy used by Kleiner\cite{kleiner} to prove Gromov's theorem, being based in particular on the study of harmonic 
functions on the Cayley graph associated to $(G,S)$.   
Kleiner's proof proceeded, roughly, along the following steps:

\begin{itemize}
\item[(i)] If $G$ is infinite amenable (or is merely without property (T)), then it admits a fixed point free 
affine action on a Hilbert space, one of its orbit maps being a
Hilbert space valued non-constant Lipschitz harmonic function on $G$.  After taking a section, this implies the existence of a scalar non-constant Lipschitz harmonic function.
\item[(ii)] If $G$ has polynomial growth, then the space of real valued Lipschitz (or fixed polynomial growth) harmonic functions on $G$ is finite-dimensional.
\item[(iii)] From (i) and (ii), an infinite group $G$ of polynomial growth admits a finite-dimensional representation with infinite image, or equivalently, a normal subgroup $H$ whose quotient $G/H$ is an infinite linear group.
\item[(iv)] Any linear group of polynomial growth is virtually solvable.
\item[(v)] From (iii) and (iv), an infinite group $G$ of polynomial growth admits a normal subgroup $H$ with virtually solvable -- or better yet 
virtually infinite cyclic --
quotient  $G/H$.
\item[(vi)] The group $H$ obtained in (v) as the kernel of a $\Z$-target homomorphism of a finite index subgroup of $G$ has slower growth, and 
thus (by induction) is virtually nilpotent. Hence $G$ is virtually solvable (in
fact, polycyclic).
\item[(vii)] Any virtually polycyclic (or merely solvable) group of 
polynomial growth is virtually nilpotent.
\end{itemize}

The idea is then to make the proofs of the statements (i)-(vii) as elementary as possible, so that they may be made quantitative. Note that
for these purposes, statements involving 
infinite objects
must be re-proved in a finitary version, and in the fully effective results
(e.g. Theorem~\ref{polycyclic})
only finitely many elements of the group are actually involved. This is 
a rather non-conventional difficulty in geometric group 
theory, where Gromov's theorem is a cornerstone. We next remark
on our modification of these steps.

\smallskip
Statement (i) is made quantitative in Section \ref{kleinersec-1}.  The proof given in \cite{kleiner} uses ultralimits and considerations related to Kazhdan's 
Property (T), and is thus of a qualitative nature. However, it turns out 
that one can use the spectral theory of the Laplacian and Young's inequality, to obtain a quantitative version of an existence theorem for non-constant
Lipschitz harmonic functions on all infinite groups.

\smallskip
Statement (ii), which we discuss in  Section \ref{kleinersec-2}, 
is the heart of Kleiner's approach (inspired by the related work of 
Colding-Minicozzi~\cite{cold}).
His arguments are already quite quantitative. Two noticeable (related)
 difficulties that arise when trying to quantify Kleiner's proof 
are that the scale $R$ for which a 
finite dimensional space of harmonic
functions injects into the ball $B(e,R)$ is not effective, and that in this
scale one lacks a priori lower bound on the positive value of the
determinant of an associated positive definite quadratic form.
 
\smallskip
For statement (iii), unlike the abstract setting of Kleiner's proof,
 the space we work with is the one appearing in 
Statement (ii) (modulo the constants),
in which the group operates naturally ({\it preserving the Lipschitz norm}). 
This is trivial in the qualitative world, but does require a certain amount of 
care in the quantitative setting (for instance, one needs quantitative versions of the assertion that any finite-dimensional vector space has a basis), and 
is done in Section \ref{kleinersec-3}.  

\smallskip
Statement (iv), established in Section~\ref{kleinersec-4},  
was proven in \cite{kleiner} 
(as in Gromov's original~\cite{gromov}) using the Tits alternative, 
or by a weaker variant of that alternative due to the 
first author \cite{shalom}. However, for our purposes we provide
a completely different, elementary proof, which makes crucial use of the
fact (not exploited previously) that the linear representation in (iii) 
ranges in a {\it compact} Lie group (the one preserving the Lipschitz 
norm in (iii)). Our arguments here are inspired by   
the famous Solovay-Kitaev theorem \cite{sk} in the theory of 
quantum computations, which
itself may be viewed as a variant on the well known 
Zassenhaus-Kazhdan-Margulis Theorem (cf.~\cite[Section 4.12]{Kap}).
As in the aforementioned results, we obtain a quantitative version of (iv)
based on the fact that commutators of matrices near the identity 
collapse fast towards it, and hence cannot be non-trivially accommodated
without a ``rapidly growing supply'' of group elements, which will be incompatible with the polynomial growth hypotheses.  Furthermore, the compactness of the ambient group ensures the existence of a finite index subgroup of matrices that are close enough to the identity that the previous considerations apply.

\smallskip
Statement (v) is formalized in Proposition \ref{trivcor2}, and 
follows simply by putting all the previous statements together.

\smallskip
Statements (vi), (vii) are formalized together as Proposition \ref{qmw}.  Statement (vii) is a result of Milnor\cite{milnor} and Wolf\cite{wolf}.  The arguments of Milnor quickly allow us to reduce to a polycyclic setting in which $G$ is an extension of a cyclic group by a virtually nilpotent group.  The main task, as in \cite{wolf}, is to then show that outer automorphisms of
 virtually nilpotent groups which are of polynomial growth, are necessarily virtually unipotent.  To do this, one must first eliminate the torsion from the nilpotent group, and then after taking quotients one is faced with understanding linear transformations of polynomial growth on a free abelian group $\Z^d$.  But this can be handled using known results on Mahler measure of algebraic integers.  It turns out that all these steps can be made semi-quantitative (in the sense that one does not control the index behind the modifier ``virtually''), and this is done in Sections \ref{wolfsec-1}-\ref{wolfsec-4}; this then gives a quantitative (but ineffective) proof of Theorem \ref{main-thm} by a compactness argument given in Section \ref{compact-sec}.  In Section \ref{effective} we discuss the changes needed to make the semi-quantitative proof fully effective.

\smallskip
There are quite a few applications of Gromov's theorem where an effective
version of it is of interest. We illustrate this only briefly in one geometric 
setting which was discussed in Gromov's original paper~\cite{gromov} as 
Corollary~\ref{fundamental} below.

\smallskip
Finally, as a by-product of our effort one obtains a simplified, 
particularly accessible proof of Gromov's 
original qualitative theorem. It avoids also the Tits alternative 
and any use of $p$-adic
numbers, and can be fully digested using basic background in 
linear algebra and calculus. This soft proof is presented in our companion
paper~\cite{Sh-Ta}, along with various other results and questions related
to the current work and around
the theme of Lipschitz harmonic functions on groups.

\subsection{Comparison with other work}

In the literature there seems to be only one previous effective result
related to Gromov's theorem, due to van den Dries and Wilkie~\cite{Wi-Dr}, 
which handles the subquadratic growth case $d<2$ by using clever, elementary, 
and quite combinatorial arguments.  (The case $d<1$ is trivial; see Lemma \ref{base}.)

 A related quantitative formulation of Kleiner's argument was recently given by Lee and Makarychev \cite{mak}.  They worked in the setting of a finite group $(G,S)$ obeying a doubling condition $|B_S(2R)| \leq 2^d |B_S(R)|$ at {\it all} 
scales $R > 0$, as opposed to the (weaker) assertion of being 
an $(R_0,d)$-growth group at a single scale.  Using Kleiner's method, they obtained upper bounds (roughly of the order of $\exp( O( d^2 ) )$) on the multiplicity of eigenvalues of the Laplacian, and also obtained a 
subgroup of $G$ of index bounded by $\exp(\exp(O( d^2 ) ))$ which had a homomorphic image onto a cyclic group $\Z/M\Z$ of cardinality at least $M \gg \exp(-O(d)) |G|^{\exp(-O(d^2))}$.  In the infinitary setting, the presence of a 
bounded index subgroup with a large homomorphic image can be used to locate a commutator subgroup with a reduced order of growth, to which one can apply an induction hypothesis to obtain a Gromov-like theorem.  Unfortunately, a 
technical obstruction in this finitary case is that the doubling condition 
in \cite{mak} is required at all scales (up to the diameter of $G$), 
whereas the homomorphic image only yields a growth reduction up to scale $M$ 
or so.  Our main Theorem \ref{main-thm} can be viewed as an answer to a question raised in \cite{mak}, regarding whether Kleiner's methods can be 
adapted to the finitary setting assuming a polynomial growth hypothesis rather than a doubling condition.

Another related result, of Milnor-Wolf type, was obtained recently by the second author in \cite{tao-solv}.  There, it was shown that if $G$ is 
a $(R_0,d)$-growth group which is solvable of derived length at most $l$, with $R_0$ sufficiently large depending on $l, d$, then $G$ contains a nilpotent group of step at most $s(l,d)$ and index at most $R_0^{C(l,d)}$ for 
some $s(l,d), C(l,d)$ depending only on $l$ and $d$.  This result has a much better control on the index of the nilpotent subgroup 
than Theorem \ref{main-thm}, but  it
gives a much poorer (though still effective) bound on the step, or on the 
size of $R_0$ required, and is, in addition, restricted to the solvable case. 
The methods used in that paper are quite different from those here, replying on
additive combinatorics rather than the theory of harmonic functions and linear
representations. By combining the results in \cite{tao-solv} with Theorem \ref{polycyclic2}, one can obtain a variant of Theorem \ref{main-thm} in which the bound $K(R_0,d)$ is completely effective, but the quantities $\exp(\exp(Cd^C))$ and $C^d$ appearing in that theorem are replaced by much larger (but still effective and explicit) functions of $d$.  We omit the details.

\subsection{Acknowledgments}

The authors thank Emmanuel Breulliard for valuable discussions, and the anonymous referee for corrections. The first
author was supported by ISF and NSF grants number 500/05 and DMS-0701639 resp. 
The second author is supported by a grant from the MacArthur Foundation, by NSF grant DMS-0649473, and by the NSF Waterman award. 

\section{Notation}

If $E, F$ are two subsets of a multiplicative group $G$, we use $E \cdot F$ to denote the product set $\{ ef: e \in E, f \in F\}$, and $E^{-1}$ to denote the inverse set $\{e^{-1}: e \in E \}$.  In an additive group we can similarly define the sum set $E+F$ and reflection $-E$.  We also define dilates $k \cdot E := \{ ke: e \in E \}$ for $E$ in an additive group and non-negative integers $k$.

Given a group $G$, we define the commutator $[g,h]$ of two group elements $g,h \in G$ by $[g,h]:=ghg^{-1}h^{-1}$, and the commutator of two subgroups $[H,K]$ to be the group generated by $\{ [h,k]: h \in H, k \in K \}$.  We define the \emph{lower central series} $G = G_1 \geq G_2 \geq \ldots$ by $G_1 := G$ and $G_{i+1} := [G_i, G]$, and the \emph{derived series} $G = G^{(1)} \geq G^{(2)} \geq \ldots$ by $G^{(1)} := G$ and $G^{(i+1)} := [G^{(i)}, G^{(i)}]$.  We say that $G$ is \emph{nilpotent of step at most $s$} if $G_{s+1}$ is trivial, and \emph{solvable of derived length at most $l$} if $G^{(l+1)}$ is trivial.

We write $X = O(Y)$, $X \ll Y$ or $Y \gg X$ to denote the statement that $|X| \leq CY$ for some absolute constant $C$.

\section{A compactness reduction}\label{compact-sec}

In this section we perform a compactness reduction to eliminate the role of the quantities $K(R_0,d)$ appearing in Theorem \ref{main-thm}.  This will simplify the proof substantially, at the cost of rendering the final value of $K(R_0,d)$ obtained ineffective; but see Section \ref{effective} for how one could avoid the use of compactness to obtain an effective value of $K(R_0,d)$.  

We first remove $K$ and $R$ from the definition of virtual nilpotency.

\begin{definition}[Virtual nilpotency, again]\label{virtnil-def2}  Let $s, D \geq 1$.  A finitely generated group $(G,S)$ is said to be \emph{virtually $(s,D)$-nilpotent} if it contains a finite index subgroup that is $(s,D)$-nilpotent.
\end{definition}

We now claim that Theorem \ref{main-thm} follows from the following ``semi-quantitative'' variant:

\begin{theorem}[Semi-quantitative Gromov theorem]\label{main-thm2}  Let $d > 0$ and $R_0 > 0$, and assume that
$$
R_0 \geq \exp(\exp(Cd^C))
$$
for some sufficiently large absolute constant $C$.  Then every $(R_0,d)$-growth group is virtually $(C^d,C^d)$-nilpotent.
\end{theorem}

We now give the (standard) compactness argument that lets us deduce Theorem \ref{main-thm} from Theorem \ref{main-thm2} (such argument appeared already
in Gromov's~\cite{gromov}, while the general formalism of the 
``space of marked finitely generated groups'' underlying it, was introduced
later by Grigorchuk in~\cite{Grigorchuk}). 

\begin{proof}[Proof of Theorem \ref{main-thm} assuming Theorem \ref{main-thm2}]
Let $C$ be the absolute constant in Theorem \ref{main-thm2}. Suppose Theorem \ref{main-thm} failed, then we could find $R_0, d$ obeying the specified bound (for this choice of $C$), and a sequence $(G_N,S_N)$ of $(R_0,d)$-growth groups with $N \to \infty$ such that $(G_N,S_N)$ is not $(N,N,C^d,C^d)$-virtually 
nilpotent. Observe that $|S_N| \leq |B_{S_N}(R_0)| \leq R_0^d$ is uniformly
bounded, hence by passing to a subsequence we may assume $|S_N|=k$ for all $N$ and some fixed $k$.
Identify now each $(G_N,S_N)$ as a quotient of the free group $(F_k,S)$
under the homomorphism mapping the free generators $S$ to $S_N$, and denote
the kernel by $M _N<F_k$. By diagonalization process and after 
passing to a subsequence, we may assume
that the sets $M _N$ converge, i.e. they eventually agree on every
finite subset of $F_k$. Their limit, $M$, is of course a normal
subgroup of $F_k$. Obviously $G=F_k/M$ with the projection of $S$ is a  
$(R_0,d)$-growth group, hence by  Theorem \ref{main-thm2} it  
has a finite index  $(C^d,C^d)$-nilpotent subgroup. By Remark~\ref{joe},
$(G,S)$ is $(K,K,C^d,C^d)$-subgroup for some $K$. To complete the argument it only
remains to observe the following three facts: 
\begin{enumerate}
\item The group $G$, being virtually nilpotent, is finitely presented;
\item If a sequence of marked groups converges to a finitely presented
group as above, then from some point on they are all quotients of it
(indeed, this happens when the finitely many relations of the limit group
stabilize in the sequence); and 
\item The property of being a 
$(K,K,C^d,C^d)$-nilpotent group is inherited by quotients.
\end{enumerate}
\end{proof}

For technical reasons it is convenient to modify the above definition slightly, by replacing the number $D$ of generators and the step $s$ with the \emph{Hirsch length} of the nilpotent group. Recall that the Hirsch length of a nilpotent group (or more generally, a polycyclic group) is the sum of the torsion-free ranks of the quotients in any normal series of that group with abelian quotients (e.g. for a nilpotent group, one could use the lower central series).  

\begin{definition}[Virtual nilpotency, yet again]\label{virtnil-def3}  Let $r \geq 1$.  A \emph{$r$-nilpotent group} is a finitely generated nilpotent group of Hirsch length at most $r$ (and thus step at most $r$, as well).  
A finitely generated group $(G,S)$ is said to be \emph{virtually $r$-nilpotent} if it contains a finite index subgroup that is $r$-nilpotent.
\end{definition}

By a result of Malcev~\cite{Mal2} if $G$ is nilpotent then any set $S \subset G$ whose
projection generates the abelianization $G/[G,G]$, generates all of $G$. 
Hence, a $r$-nilpotent 
group contains a finite index subgroup generated by at most $r$ generators (or $2r$, if one enforces symmetry).  Thus one may replace the conclusion of 
Theorem \ref{main-thm2} by the assertion that every $(R_0,d)$-growth group is virtually $C^d$-nilpotent.

It remains to prove Theorem \ref{main-thm2} with this modification.  This is the purpose of the remaining sections of the paper.

\section{Generator bounds}

In this section we collect a number of useful results which exploit polynomial growth hypotheses to locate bounded sets of generators for various types of groups; such results will be used frequently in the sequel.   

First, we observe that the property of having quantitative finite index is transitive.

\begin{lemma}[Transitivity of quantitative finite index]\label{trans} If $(G',S')$ is a $(K,R)$-subgroup of $(G,S)$, and $(G'',S'')$ is a $(K',R')$-subgroup of $(G',S')$, then $(G'',S'')$ is a $(KK'(K+RK'+1), RR')$-subgroup of $(G,S)$.  
\end{lemma}

\begin{proof} Clearly $S'' \subset B_{S'}(R') \subset B_S(RR')$.  Next, since
$$ B_S(K+1) \subset B_S(K) \cdot B_{S'}(K)$$
we have
$$ B_S(n) \subset B_S(K+n) \subset B_S(K) \cdot B_{S'}(nK)$$
for any $n \geq 1$.  Similarly
$$ B_S(n') \subset B_{S'}(K'+n') \subset B_{S'}(K') \cdot B_{S''}(n' K')$$
for all $n' \geq 1$.  Combining the two, we conclude that
$$ B_S(n) \subset B_S(K) \cdot B_{S'}(K') \cdot B_{S''}(nKK')$$
for all $n \geq 1$.  Since $B_{S'}(K') \subset B_S(RK')$, we conclude that
$$ B_S(n) \subset B_S(K+RK') \cdot B_{S''}(nKK').$$
Setting $n := K+RK'+1$, we conclude that
$$ B_S(K+RK'+1) \subset B_S(K+RK') \cdot B_{S''}(KK'(K+RK'+1)).$$
and thus
$$ B_S(KK'(K+RK'+1)+1) \subset B_S(KK'(K+RK'+1)) \cdot B_{S''}(KK'(K+RK'+1))$$
and the claim follows.
\end{proof}

Next, we observe that a $(R_0,d)$ group has a bounded number of generators, after passing to a finite index subgroup.

\begin{lemma}[Generator reduction of $(R_0,d)$-growth groups]\label{gen-red}  Let $d \geq 1$, $0 < \kappa < 1$, and $R_0 \geq 100^{1/\kappa}$, and let $(G,S)$ be a $(R_0,d)$-growth group.  Then there exists a $(R_0^\kappa, R_0^\kappa)$-subgroup $(G',S')$ of $G$ with $|S'| \ll O(1)^{d/\kappa}$.  In particular (by Remark \ref{growth}) $(G',S')$ is a $(R_0^{1-\kappa}, d/(1-\kappa))$-growth group.  Furthermore we have $G = G' \cdot B_S(R_0^\kappa)$.
\end{lemma}

\begin{remark} In practice, we will take $\kappa$ to be small compared to $d$ (e.g. $\kappa = 1/100d$), and so the slight degradation $d \mapsto d/(1-\kappa)$ in the order of growth here will be acceptable (at a later stage of the argument, we will obtain a reduction in the growth order by $1$, which will more than compensate for these sorts of losses).
\end{remark}

\begin{proof}  Clearly
$$ 1 \leq |B_S(1)| \leq |B_S(R_0^\kappa)| \leq |B_S(R_0)| \leq R_0^d.$$
By the pigeonhole principle, one can thus find a radius $1 \leq r \leq R_0^\kappa / 10$ such that
$$ |B_S(10r)| \ll O(1)^{d/\kappa} |B_S(r)|.$$
Fix this $r$.  Now, let $X$ be a maximal subset of $B_S(4r)$ such that the sets $x \cdot B_S(r)$ for $x \in X$ are disjoint.  Since the $x \cdot B_S(r)$ are contained in $B_S(5r)$, we have
$$ |X| \leq |B_S(5r)|/|B_S(r)| \ll O(1)^{d/\kappa}.$$
On the other hand, by construction of $X$ we have the covering property
$$ B_S(4r) \subset X \cdot B_S(2r).$$
In particular, if we set $S' := X \cup X^{-1}$, and let $G'$ be the group generated by $S'$, then $|S'| \ll O(1)^{d/\kappa}$ and
\begin{equation}\label{bs4r}
B_S(4r) \subset S' \cdot B_S(2r);
\end{equation}
$$ B_S(nr) \subset B_{S'}(n) \cdot B_S(2r)$$
for $n=1,2,\ldots$.  In particular, 
we have
$$ B_S(R_0^\kappa+1) \subset B_{S'}(R_0^\kappa) \cdot B_S(R_0^\kappa)$$
which on inversion gives
$$ B_S(R_0^\kappa+1) \subset B_S(R_0^\kappa) \cdot B_{S'}(R_0^\kappa).$$
On the other hand, we have $S' \subset B_S(4r) \subset B_S(R_0^\kappa)$, and the claim follows.
\end{proof}

\begin{remark}
Because of this lemma, we will be able to safely absorb a number of terms involving the size $|S|$ of the generating set in the arguments that follow.  (One should think of $d$ and $1/\kappa$ as being bounded; the key point is that the bound on $|S|$ is independent of $R_0$.)   
\end{remark}

\begin{remark}  If one does not pass to a finite index subgroup, then one may need as many as $\log R_0$ generators.  Indeed, consider the abelian group $G = \Z_2^n \times \Z$, where $n \sim \log R_0$.  Then $B_S(R_0)$ is of polynomial size in $R_0$, but one needs $n+1 \sim \log R_0$ generators in order to generate the whole group $G$.  In the converse direction, if $B_S(R_0) \leq R_0^d < 2^{R_0}$, then if we let $s_1,\ldots,s_n$ be a maximal sequence in $S$ which is \emph{dissociated} (i.e. the words $s_1^{i_1} \ldots s_n^{i_n}$ for $i_1,\ldots,i_n \in \{0,1\}$ are different), then one easily verifies that $2^n \leq |B_S(R_0)| \leq R_0^d < 2^{R_0}$ and thus $n \leq d \log_2 R_0 < R_0$, and that the set $\{s_1,\ldots,s_n,s_1^{-1},\ldots,s_n^{-1}\}$ generates $S$ and thus $G$.  Thus $G$ can be generated by $O( d \log R_0)$ generators in this case.
\end{remark}

\begin{remark} In many applications of Theorem \ref{main-thm}, $S$ would already be bounded.  But in our proof of Theorem \ref{main-thm} (which is based on an induction on $d$), it will become necessary at some stage in the proof to pass from $G$ to a subgroup such as $[G,G]$, which need not have a bounded number of generators.  It is then that Lemma \ref{gen-red} becomes necessary. 
\end{remark}

We shall need two further results in a similar spirit.  The first 
asserts that if a $(R_0,d)$-growth group $G$ can be generated by a small number of generators and has polynomial growth, then so does $[G,G]$ 
(cf.~\cite[page 61]{gromov}, or~\cite{milnor}):

\begin{lemma}[Generator reduction for a commutator group]\label{zonk} Let $d \geq 1$, $0 < \kappa < 1$, and $R_0 \geq C d^{C/\kappa}$ for some sufficiently large absolute constant $C$, and let $(G,S)$ be a $(R_0,d)$-growth group.  Then there exists a set of generators $S'$ of $[G,G]$, each of the form $g[e,e']g^{-1}$ for some $e,e' \in S$ and $g \in B_S(\frac{1}{2} R_0^\kappa - 4)$.  In particular, $([G,G],S')$ is a $(\infty,R_0^\kappa)$-subgroup of $(G,S)$, and thus (by Remark \ref{growth}) $([G,G],S')$ is a $(R_0^{1-\kappa}, d/(1-\kappa))$-growth group.
\end{lemma}

\begin{proof}
It is not difficult to see that $[G,G]$ is generated by the elements $g [e,e'] g^{-1} = [geg^{-1}, ge'g^{-1}]$ where $e,e' \in S$ and $g \in G$ (since modulo the normal group generated by these elements, all the basis elements $e, e'$ commute).  The difficulty is to replace this infinite generating set by a finite one.

For any $r \geq 1$, let $A_r := \{ g[e,e'] g^{-1}: e,e' \in S, g \in B_S(r)\} \subset B_S(2r+4)$, and let $A_{\leq r} := A_r \cdot A_{r-1} \cdot \ldots \cdot A_0$.  Then $A_{\leq r} \subset B_S(O(r+1)^2)$, and thus by hypothesis we have $|A_{\leq r}| \leq R_0^d$ for $r \leq c R_0^{1/2}$ and some small absolute constant $c>0$.  By hypothesis (and if $C$ is large enough), $R_0^{\kappa/2}$ is larger than a large multiple of $d \log R_0$.  In particular, one can find a $2 \leq r_0 \leq \frac{1}{2} R_0^\kappa - 4$ such that
$$ |A_{\leq r_0}| < 2^{r_0-2}.$$
From the pigeonhole principle, we can thus find $2 \leq r \leq r_0$ such that 
\begin{equation}\label{ar}
|A_{\leq r+1}| < 2 |A_{\leq r}|.
\end{equation}
This implies that for any $x \in A_{r+1}$, that $x \cdot A_{\leq r}$ and $A_{\leq r}$ overlap, thus 
\begin{equation}\label{aer}
A_{\leq r+1} \subset A_{\leq r} \cdot A_{\leq r}^{-1}.
\end{equation}
In particular, $A_{r+1}$ is contained in the group generated by $A_r$, which implies that the act of conjugation by any element $g \in S$ preserves the group generated by $A_r$.  Since $A_r$ also contains the commutators $[e,e']$ for $e,e' \in S$, we conclude that $S' := A_r$ generates $[G,G]$; also, $S'$ is symmetric by construction.   The claims of the lemma then follow.
\end{proof}

Next, we show that finite index subgroups of finitely generated groups continue to be finitely generated in a very quantitative manner.

\begin{lemma}[Qualitative finite index implies quantitative finite index]\label{index}  Let $d \geq 1$, $I \geq 1$.  Let $G = (G,S)$ be a finitely generated group, and let $G'$ be a finite index subgroup of $G$ with index $|G:G'| \leq I$.  Then there exists a set $S' \subset B_S(2I+1)$ of generators of $G'$ such that $(G',S')$ is a $(2I+1,2I+1)$-subgroup of $G$.
\end{lemma}

\begin{proof}  For each $r > 0$, the set $B_S(r) \cdot G'$ is a union of left cosets of $G'$.  The number of such cosets is of course $|G:G'|$.  Thus by the pigeonhole principle, one can find $0 \leq r_0 \leq I$ such that $B_S(r_0+1) \cdot G' = B_S(r_0) \cdot G'$.  In particular,
$$ B_S(r_0+1) \subset B_S(r_0) \cdot G'.$$
If we set $S' := G' \cap B_S(2r_0+1)$, then $S'$ is symmetric and contained in $B_S(2I+1)$, and by the triangle inequality we have
$$ B_S(r_0+1) \subset B_S(r_0) \cdot S'.$$
Multiplying both sides on the left by $B_S(r-r_0)$, we obtain 
$$
B_S(r+1) \subset B_S(r) \cdot S'
$$
for all $r \geq r_0$ (and hence for all $r \geq 2I+1$).  In particular
$$ 
B_S(2I+2) \subset B_S(2I+1) \cdot B_{S'}(2I+1)$$
and the claim follows.
\end{proof}

\section{Reduction to two key propositions}\label{reduct-sec}

Theorem \ref{main-thm2} (modified as discussed at the end of Section \ref{compact-sec}) is deduced from one easy proposition and two 
difficult ones.  We begin with the easy proposition, which handles the base case $d<1$:

\begin{lemma}[Sublinear growth case]\label{base}  Suppose that $(G,S)$ is a $(R,d)$-growth group for some $0 < d < 1$ and $R > 1$.  Then $G = B_S(R)$ and so $G$ is finite with $|G| \leq R^d$.
\end{lemma}

\begin{proof}  The $|B_S(r)|$ for $0 \leq r < R$ are integers between $1$ and $|B_S(R)| = R^d < R$ that increase in $r$.  Thus by the pigeonhole principle, there exists $0 \leq r < R-1$ such that $|B_S(r)|=|B_S(r+1)|$, thus $B_S(r) = B_S(r+1)$.  Iterating this we see that $B_S(r) = B_S(r')$ for all $r' \geq r$, thus $G = B_S(r) = B_S(R)$ and the claim follows.
\end{proof}

The higher order growth cases $d \geq 1$ will then be handled by an induction on $d$.  The first major step in this process is to obtain a reduction in the order of growth for a certain commutator subgroup:

\begin{proposition}[Reduction in growth order]\label{trivcor2}  Let $R_0, d \geq 1$, and let $(G,S)$ be a $(R_0,d)$-growth group.  Assume that
\begin{equation}\label{rdc}
R_0 \geq \exp(\exp(Cd^C))
\end{equation}
for some sufficiently large absolute constant $C$.  Then at least one of the following holds:
\begin{itemize}
\item[(i)] $G = B_S(R_0^{\exp(\exp(O(d)^{O(1)}))})$ (in particular, $G$ is finite).
\item[(ii)] There exists a $(R_0^{1/10},R_0^{1/10})$-subgroup $(G',S')$ of $G$ and a positive integer $l = O(d^{O(1)})$ such that $(G')^{(l)}$ is generated by a set $(S')^{(l)} \subset B_{S'}(R_0^{1/10})$ for which $( (G')^{(l)}, (S')^{(l)})$ is a $(R, d-0.9)$-growth group for some $R_0^{1 - 1/10d} \leq R \leq R_0$.
\end{itemize}
\end{proposition}

This proposition will be established in Sections \ref{kleinersec-1}-\ref{kleinersec-5} using (quantitative versions of) some arguments of Kleiner \cite{kleiner}.  The bounds here are completely effective (and far superior to those in Proposition \ref{qmw} below).

To exploit this growth reduction we need the following variant of Theorem \ref{main-thm}.

\begin{proposition}[Semi-quantitative Milnor-Wolf theorem]\label{qmw}  Let $s, l, r, R_0, d \geq 1$, and suppose that
$$ R_0 \geq r^C (Cd)^{Cdl}$$
for some sufficiently large absolute constant $C$.  Suppose we have a short exact sequence
$$ 0 \to H \to G \to L \to 0$$ 
of groups, where $G = (G,S)$ is a $(R_0,d)$-growth group, $H = (H,S_H)$ is a virtually $r$-nilpotent (see Def.~\ref{virtnil-def3}) $(\infty,R_0^{1/10})$-subgroup of $G$, and $L$ is solvable of derived length at most $l$.  Then $(G,S)$ is virtually $r + l O(1)^d$-nilpotent.
\end{proposition}

We prove this proposition in Sections \ref{wolfsec-1}-\ref{wolfsec-4}.  The arguments are quite different from those used to prove Proposition \ref{trivcor2}, being related to the arguments used by Milnor \cite{milnor} and Wolf \cite{wolf}.  A key new technical difficulty in this ``single-scale'' setting, not present in earlier ``multi-scale'' work, is that the index of the nilpotent subgroup of $H$ may be so large that this subgroup is not ``visible'' at the one scale $R_0$ that we directly control.

In the remainder of this section we show how the above three propositions imply Theorem \ref{main-thm2} (modified as discussed at the end of Section \ref{compact-sec}).  From Lemma \ref{base} we see that Theorem \ref{main-thm2} holds for $d < 1$.  We may thus assume inductively that $d \geq 1$ and that the claim has already been proven for $d-0.9$.  By Proposition \ref{trivcor2}, we see that either conclusion (i) or conclusion (ii) of that proposition holds.  If conclusion (i) holds, then $G$ is finite and the claim follows.  If instead conclusion (ii) holds, then by induction hypothesis we see that $( (G')^{(l)}, (S')^{(l)})$ is virtually $C^{d-0.9}$-nilpotent.  On the other hand, $G' / (G')^{(l)}$ is clearly solvable of derived length at most $l = O(d^{O(1)})$.
Applying Proposition \ref{qmw} to the short exact sequence
$$ 0 \to (G')^{(l)} \to G' \to G' / (G')^{(l)} \to 0$$
we conclude (if $C$ is large enough) that $G'$ is virtually $C^d$-virtually nilpotent.  Since $G'$ is a a finite index subgroup of $G$, the claim then follows.

\section{First step for Proposition \ref{trivcor2}: produce a non-trivial Lipschitz almost harmonic function}\label{kleinersec-1}

We now begin the proof of Proposition \ref{trivcor2}, following Kleiner \cite{kleiner}.
Kleiner's argument is based on harmonic functions, and in particular on studying the class of functions $u: G \to \R$ which are both Lipschitz and harmonic with respect to the set $S$ of generators.  In our quantitative applications, the group $G$ could well be finite (with extremely large cardinality), and so there may not be any harmonic functions other than the constants.  To deal with this we need to consider instead a somewhat larger class of \emph{almost harmonic functions}.  A good example is provided on the finite group $G = \Z/N\Z$ (with the standard generators $+1,-1$, and with $N$ large) by the function $u(x) := N \sin(2\pi x/N)$.  This function is Lipschitz (with Lipschitz norm approximately $2\pi$), and ``almost harmonic'' in the sense that $u(x+1) + u(x-1) - 2u(x) = O(1/N)$ for all $x$. 

It is known (see~\cite[Appendix]{kleiner} and the references therein) 
that infinite groups admit non-trivial Lipschitz harmonic functions.  
It is thus not surprising that all ``large'' groups (finite or infinite) 
admit non-trivial \emph{almost} harmonic Lipschitz functions.  
To see this, let us first formalize our definitions.

\begin{definition}  Let $u: G \to \R$ be a function.  The \emph{gradient} $\nabla u: G \to \R^S$ of a $u$ is defined by the formula
$$ \nabla u(x) := (u(xs) - u(x))_{s \in S}$$
so in particular
$$ |\nabla u(x)| := (\sum_{s \in S} |u(xs) - u(x)|^2)^{1/2}.$$
Dually, given a vector-valued function $F = (F_s)_{s \in S}: G \to \R^S$, we can define its \emph{divergence} $\nabla \cdot F: G \to \R$ by the formula
$$ \nabla \cdot (F_s)_{s \in S}(x) = \sum_{s \in S} F_s(xs^{-1}) - F_s(x).$$
The \emph{Laplacian} of a function $u: G \to \R$ is defined by the formula
$$ \Delta u := -\nabla \cdot \nabla u$$
or more explicitly (using the symmetry of $S$)
$$ \Delta u(x) = 2 |S| u(x) - 2 \sum_{s \in S} u(x s).$$
The \emph{Lipschitz (semi)-norm} $\|u\|_{\Lip}$ of $u$ is defined as
$$ \| u \|_{\Lip} := \| \nabla u \|_{\ell^\infty(G)} = \sup_{x \in G} |\nabla u(x)|.$$
A function $u: G \to \R$ is \emph{harmonic} if $\Delta u = 0$.
If $\eps > 0$, we say that a function $u: G \to \R$ is \emph{$\eps$-harmonic Lipschitz} if $\|u\|_{\Lip} \leq 1$ and $\|\Delta u\|_{\ell^\infty(G)} \leq \eps$.
\end{definition}

\begin{remark} The conventions for the Laplacian $\Delta$ and Lipschitz norm may differ by some (ultimately irrelevant) constant factors from other definitions in the literature.  With these conventions, the function $\frac{1}{2\pi} N \sin(2\pi x/N)$ is $O(1/N)$-harmonic Lipschitz on $\Z/N\Z$.
\end{remark}

\begin{proposition}[Infinite groups have non-constant Lipschitz harmonic functions]\label{infilip}  Let $R \geq 1$, and let $(G,S)$ be a finitely generated group.  Then at least one of the following statements is true:
\begin{itemize}
\item $G = B_S(R)$.
\item There exists a $O(|S| R^{-1/3})$-harmonic Lipschitz function $u: G \to \R$ with $|\nabla u(\id)| \geq 1/|S|$.
\end{itemize}
\end{proposition}

\begin{remark} Comparing this with the example $u(x) = \frac{1}{2\pi} N \sin(2\pi x/N)$ on $\Z/N\Z$ we see that the exponent $-1/3$ in the Proposition should probably be a $-1$.  However, any positive exponent will suffice for our application.
\end{remark}

\begin{proof} 
Given two functions $u, v: G \to \R$, we formally define the \emph{convolution} $u*v: G \to \R$ by the formula
$$ u*v(x) := \sum_{y \in G} u(y) v(y^{-1} x) = \sum_{y \in G} u(x y^{-1}) v(y).$$
By Young's inequality, convolution is well-defined whenever $u \in \ell^p(G), v \in \ell^q(G)$ and $1/p+1/q=1/r+1$ for some $1 \leq p,q,r \leq \infty$, in which case we have
\begin{equation}\label{young}
\|u*v\|_{\ell^r(G)} \leq \|u\|_{\ell^p(G)} \|v\|_{\ell^q(G)}.
\end{equation}
We observe the pleasant identities
\begin{align}
\nabla( f_1 * f_2 ) &= f_1 * \nabla f_2 \label{nabla-move}\\
\nabla \cdot ( f * F ) &= f * (\nabla \cdot F) \label{div-move}
\end{align}
and thus
\begin{equation}\label{harm}
\Delta(f_1 * f_2) = f_1 * \Delta f_2.
\end{equation}

Fix $R$.  We may assume that $G \neq B_S(R)$, since are done otherwise.  Let $\sigma$ be the measure $\sigma = \frac{1}{|S|} \sum_{s \in S} \delta_s$, where $\delta_s$ is the Kronecker delta.  Observe that $\Delta u = 2|S| ( u * \sigma - u )$ for any $u$, where $u * \sigma(x) := \sum_y u(xy^{-1}) \sigma(y)$ is the usual convolution operator.  

Let $\sigma^{(m)} := \sigma * \ldots * \sigma$ be the $m$-fold convolution of $\sigma$, and let $f: G \to \R$ be the function
$$ f := \frac{1}{R} \sum_{m=0}^{R} \sigma^{(m)}.$$
Then we have $\|f\|_{\ell^1(G)} = 1$ and $\|\Delta f \|_{\ell^1(G)} \ll |S| / R$.  We divide into two cases.

{\bf Case 1 (``non-amenable'' case):} $\| \nabla f \|_{\ell^1(G)} \geq R^{-2/3}$.  By the pigeonhole principle, we can then find $s \in S$ such that the function $f_s(x) := f(xs) - f(x)$ has $\ell^1$ norm at least $|S|^{-1} \| \nabla f \|_{\ell^1(G)}$.  If we then define $u := h * f$, where $h(x) := \sgn(f_s(x^{-1})) / \| \nabla f \|_{\ell^1(G)}$, we see from \eqref{young}, \eqref{nabla-move} that
$$ \| \nabla u \|_{\ell^\infty(G)} \leq 1$$
and
$$ |u(s)-u(\id)| \geq 1/|S|$$
while from \eqref{young}, \eqref{harm} we have
$$ \| \Delta u \|_{\ell^\infty(G)} \ll (|S|/R) / R^{-2/3}$$
and the claim follows.

{\bf Case 2. (``amenable'' case):} $\| \nabla f \|_{l^1(G)} < R^{-2/3}$.  If we set $F := |f|^{1/2}$ then we have the pointwise bound
$$ |\nabla F| \leq |S|^{1/2} |\nabla f|^{1/2}$$
thanks to the elementary estimate $|x|^{1/2} - |y|^{1/2} \leq |x-y|^{1/2}$ for $x,y \geq 0$.  Thus we have
$$ \|F\|_{\ell^2(G)} = 1 \hbox{ and } \| \nabla F \|_{\ell^2(G)} \leq |S|^{1/2} R^{-1/3}.$$
Also, $F$ is supported on $B_S(R)$.  Since $G \neq B_S(R)$ by hypothesis, $F$ cannot be constant.  In particular
$$ \|\nabla F\|_{\ell^2(G)} > 0.$$
Thus if $\mu_F$ is the spectral measure of $\Delta$ relative to $F$, we see that $\mu_F$ is a probability measure on $[0,1]$ with
$$ 0 < \int_0^1 x\ d\mu_F(x) = \| \nabla F \|_{\ell^2(G)}^2 \leq |S| R^{-1/3}.$$
This implies that 
$$ 0 < \int_0^{|S| R^{-1/3}} x^2\ d\mu_F(x) \leq |S| R^{-1/3} \int_0^{|S| R^{-1/3}} x\ d\mu_F(x).$$
If we then let $F' \in \ell^2(G)$ be the spectral projection of $F$ to $[0,|S| R^{-1/3}]$, we conclude that 
\begin{equation}\label{fp}
0 < \|\Delta F'\|_{\ell^2(G)}^2 \leq |S| R^{-1/3} \|\nabla F'\|_{\ell^2(G)}^2.
\end{equation}
By the pigeonhole principle, we can find $s \in S$ such that the function $F'_s(x) := F'(xs) - F'(x)$ obeys the lower bound
\begin{equation}\label{fh}
 \| F'_s\|_{\ell^2(G)}^2 \geq |S|^{-1} \|\nabla F'\|_{\ell^2(G)}^2.
\end{equation}
If we define $u := H * F$, where $H(x) := F'_s(x^{-1}) / \|\nabla F \|_{\ell^2(G)}^2$, we see from \eqref{young}, \eqref{div-move} that
$$ \| \nabla u \|_{\ell^\infty(G)} \leq 1$$
while from \eqref{fh} we have
$$ |u(s)-u(\id)| \geq 1/|S|$$
and from \eqref{young}, \eqref{harm}, \eqref{fp} we have
$$ \| \Delta u \|_{\ell^\infty(G)} \leq |S| R^{-1/3} $$
and the claim follows.
\end{proof}

\noindent {\bf Remark.} See also the recent~\cite{lp} for a related 
probabilistic construction of (Hilbert space valued) harmonic maps on amenable groups. 

Once one has one non-trivial almost harmonic function $u$, one can then create a further family $\rho(g) u$, $g \in G$ of almost harmonic functions by translation: $\rho(g) u(x) := u(g^{-1} x)$.  Kleiner's approach to Gromov's theorem revolves around a study of this family, starting with the fundamental observation that this family is (approximately) finite-dimensional.  We now turn to this important fact.

\section{Second step for Proposition \ref{trivcor2}: Kleiner's theorem}\label{kleinersec-2}

Kleiner\cite{kleiner}, building upon work of Colding and Minicozzi\cite{cold}, proved that in a group of polynomial growth the space of Lipschitz harmonic functions is finite-dimensional; indeed, an inspection of the argument shows that the dimension of this space is bounded by some constant $C(d)$ depending only on the order $d$ of growth.  The main objective of this section is to prove the following quantitative version of this fact:

\begin{theorem}[Quantitative Kleiner theorem]\label{qkt}  There exists an absolute constant $C > 0$ such that for any $0 < \kappa < 0.1$,
and any $(R_0,d)$-growth group $(G,S)$ with $d \geq 1$ and
\begin{equation}\label{rbit}
 R_0 \geq K := (C|S|)^{C d^3/\kappa^2},
\end{equation}
there exists a finite-dimensional subspace $V$ of $\R^G$ of dimension
$$ \dim(V) = O(|S|)^{O(d^3/\kappa^2)}$$
such that for every $R_0^{-K}$-harmonic Lipschitz function $u: G \to \R$, there exists $v \in V$ such that
$$ \| u-v \|_{\ell^2(B_S(R_0^{1-\kappa}))} \leq R_0^{-100d}.$$
\end{theorem}

We will prove Theorem \ref{qkt} in several stages, following \cite{kleiner}.  For any $x \in G$ and $r \geq 1$, we write $B(x,r) := x \cdot B_S(r)$ for the ball of radius $r$ centered at $x$.  The first step is the following 
Poincar\'e inequality (compare with one first proved by Colding and 
Minicozzi in~\cite{cold}, and its adaptation to the group setting 
due to Kleiner-Saloff-Coste in~\cite[Theorem 2.2]{kleiner}):

\begin{lemma}[Poincar\'e inequality]\label{poincare}  Let $f: G \to \R$, $x \in G$, and $r \geq 1$. Let $f_{B(x,r)} := \frac{1}{|B(x,r)|} \int_{B(x,r)} f$ be the average value of $f$ on $B(x,r)$.  Then
$$ \| f - f_{B(x,r)} \|_{\ell^2(B(x,r))} \leq 2r \frac{|B_S(2r)|}{|B_S(r)|} \| \nabla f \|_{\ell^2(B(x,3r))}.$$
\end{lemma}

\begin{proof}  By definition of the gradient, we have the pointwise bound
$$ |f(ygs) - f(yg)| \leq |\nabla f(yg)|$$
for all $y,g \in G$ and $s \in S$.  If we take $g \in B_S(2r)$ and average this in $\ell^2$ over all $y \in B(x,r)$, we conclude that
$$ (\sum_{y \in B(x,r)} |f(ygs) - f(yg)|^2)^{1/2} \leq \| \nabla f \|_{\ell^2(B(x,3r))}.$$
Telescoping this using the triangle inequality, we conclude that
$$ (\sum_{y \in B(x,r)} |f(yg) - f(y)|^2)^{1/2} \leq 2r \| \nabla f \|_{\ell^2(B(x,3r))}.$$
for all $g \in B_S(2r)$.  Summing in $g$ using the triangle inequality, we conclude that
$$ (\sum_{y \in B(x,r)} (\sum_{g \in B_S(2r)} |f(yg) - f(y)|)^2)^{1/2} \leq 2r |B_S(2r)| \| \nabla f \|_{\ell^2(B(x,3r))}.$$
But for any $y \in B(x,r)$, we have
$$ |f(y) - f_{B(x,r)}(y)| \leq \frac{1}{|B_S(r)|} \sum_{z \in B(x,r)} |f(z)-f(y)|
\leq \frac{1}{|B_S(r)|} \sum_{g \in B_S(2r)} |f(yg)-f(y)|$$
and the claim follows.
\end{proof}

For $\eps$-harmonic functions, we have a reverse inequality:

\begin{lemma}[Reverse Poincar\'e inequality]\label{rev}  Let $f: G \to \R$ be $\eps$-harmonic Lipschitz, and let $B(x,r)$ be a ball for some $r \geq 1$.  Then
$$ \| \nabla f \|_{\ell^2(B(x,r))} \ll |S|^{O(1)} (\frac{1}{r} \|f\|_{\ell^2(B(x,2r))} + \eps r |B_S(2r)|^{1/2}).$$
\end{lemma}

\begin{proof}  Let $\psi: G \to \R$ be the cutoff function $\psi(y) := \max( 1 - \frac{\dist(x,y)}{2r}, 0 )$.  Then it will suffice to show that
\begin{equation}\label{gnaf}
 \sum_{y \in G} |\nabla f|^2 \psi^2(y) \ll |S|^{O(1)} (\frac{1}{r^2} \|f\|_{\ell^2(B(x,2r))}^2 + \eps^2 r^2 |B_S(2r)|).
\end{equation}
We may clearly restrict the sum on the left to $B(x,2r-1)$.  Now, for any $y \in B(x,2r-1)$ and $s \in S$ we have
\begin{equation}\label{fsp}
 (f(ys) - f(y)) \psi^2(y) = (f \psi^2(ys) - f\psi^2(y)) - f(ys) \psi(y) (\psi(ys) - \psi(y)) - f(ys) (\psi(ys) - \psi(y))^2.
\end{equation}
From the triangle inequality, $\dist(x,ys)$ differs from $\dist(x,y)$ by at most $1$, and thus
$$ \psi(ys) - \psi(y) = O(1/r).$$
Also, $f(ys) = O( |f(y)| + |\nabla f(y)| )$.  Multiplying \eqref{fsp} by $f(ys)-f(y)$ and summing in $s$, we conclude that
$$ |\nabla f|^2 \psi^2(y) = \nabla(f \psi^2) \cdot \nabla f + O( |S| (|f(y)| + |\nabla f(y)|) (\frac{\psi(y)}{r} + \frac{1}{r^2}) ).$$
Inserting this into the left-hand side of \eqref{gnaf} and summing by parts, we conclude that
\begin{equation}\label{paris}
\begin{split}
&  \sum_{y \in G} |\nabla f|^2 \psi^2(y) \ll \sum_{y \in G} |f\psi^2(y)| |\Delta f(y)|
+ \frac{|S|}{r} \sum_{y \in G} |\nabla f(y)| (|f(y)| \\
&\quad  + |\nabla f(y)|) \psi(y)
+ \frac{|S|}{r^2} \sum_{y \in B(x,2r-1)} |\nabla f(y)|(|f(y)| + |\nabla f(y)|).
\end{split}
\end{equation}
As $f$ is $\eps$-harmonic, we can use Cauchy-Schwarz to bound
\begin{align*}
\sum_{y \in G} |f\psi^2(y)| |\Delta f(y)| &\leq \eps \sum_{y \in B_S(2r-1)} |f(y)| \\
&\leq \eps |B_S(2r)|^{1/2} \|f\|_{\ell^2(B_S(x,2r-1))} \\
&\leq \eps^2 r^2 |B_S(2r)| + \frac{1}{r^2} \|f\|_{\ell^2(B_S(x,2r-1))}^2.
\end{align*}
Another application of Cauchy-Schwarz gives
$$ \frac{|S|}{r} \sum_{y \in G} |\nabla f(y)| (|f(y)| + |\nabla f(y)|) \psi(y)
\leq c \sum_{y \in G} |\nabla f(y)|^2 \psi^2(y) + \frac{1}{c} \frac{|S|^2}{r^2} \sum_{y \in G} (|f(y)| + |\nabla f(y)|)^2$$
for any $c > 0$.  By choosing $c$ small enough, we can then absorb the first term on the right-hand side into the left-hand side of \eqref{paris} and conclude that
$$  \sum_{y \in G} |\nabla f|^2 \psi^2(y) \ll \eps^2 r^2 |B_S(2r)|
+ \frac{|S|^{O(1)}}{r^2} ( \|f\|_{\ell^2(B_S(x,2r-1))}^2 + \| \nabla f \|_{\ell^2(B_S(x,2r-1))}^2).$$
But from definition of $\nabla$ and the triangle inequality we see that
$$ \| \nabla f \|_{\ell^2(B_S(x,2r-1))} \ll |S| \| f \|_{\ell^2(B_S(x,2r))}$$
and the claim follows.
\end{proof}

Now we need some more definitions.  Given $R > 0$, define the symmetric bilinear form $Q_R: \R^G \times \R^G \to \R$ by the formula
\begin{equation}\label{qr}
 Q_R( u , v ) := \sum_{x \in B_S(R)} (u(x)-u(\id)) (v(x)-v(\id)).
\end{equation}
This form is clearly positive semi-definite (with finite rank).  Given any $u_1,\ldots,u_k: G \to \R$, we define the $R$-volume $\Vol_R(u_1,\ldots,u_k)$ by the formula
$$ \Vol_R(u_1,\ldots,u_k) = \det( (Q_R(u_i,u_j))_{1 \leq i,j \leq k} )^{1/2};$$
the right-hand side is non-negative due to the positive semi-definite nature of $Q_R$.
Geometrically, one can view $\Vol_R(u_1,\ldots,u_k)$ as the length of the wedge product $u_1 \wedge \ldots \wedge u_k$ in $\bigwedge^k \R^G$, using the induced semi-definite form from $Q_R$.  

The $Q_R$ are monotone increasing in $R$ as symmetric bilinear forms, which implies the monotonicity relationship
\begin{equation}\label{volmono}
 \Vol_R(u_1,\ldots,u_k) \leq \Vol_{4R}(u_1,\ldots,u_k).
\end{equation}

When $k$ is large, and the $u_1,\ldots,u_k$ are approximately harmonic, we can improve this inequality by applying Proposition \ref{rev} on large balls and Proposition \ref{poincare} on small balls.  More precisely, we have the following inequality (cf. \cite[Lemma 3.16]{kleiner}, \cite[Proposition 4.16]{cold}):

\begin{proposition}[Volume decrease]\label{codim}  Let $k \geq 1$ be an integer, $0 < \eps, \delta < 1$ and $R > 1/\delta$.  Let $u_1,\ldots,u_k: G \to \R$ be $\eps$-harmonic Lipschitz functions.  Suppose also that
\begin{equation}\label{kd}
 k \geq 2 \left(\frac{|B_S(2R)|}{|B_S(\delta R)|} + 1\right).
\end{equation}
Then we have
$$ \Vol_R(u_1,\ldots,u_k) \leq
O(|S|)^{O(k)} \left( \delta^{k/2}
\left(\frac{|B_S(7 \delta R)|}{|B_S(\delta R)|}\right)^{k/2}
\Vol_{4R}(u_1,\ldots,u_k)
+ k \eps^2 R^{k+2} |B_S(4R)|^{k/2}\right).
$$
\end{proposition}

\begin{proof} Let us write
\begin{equation}\label{vdef}
v := \Vol_{4R}(u_1,\ldots,u_k).
\end{equation}
We may of course assume that 
\begin{equation}\label{vbig}
v > k \eps^2 R^{k+2} |B_S(4R)|^{k/2},
\end{equation}
otherwise we are done by \eqref{volmono}.  In particular, this implies that $u_1,\ldots,u_k$ are linearly independent.  Let $V$ be the $k$-dimensional subspace of $\R^G$ spanned by the $u_1,\ldots,u_k$; the non-vanishing of $v$ implies that $V$ is a Hilbert space with respect to the bilinear form $Q_{4R}$.

We will now study a subspace $V'$ of $V$ consisting of ``locally mean zero'' functions. Let $x_1, \ldots, x_m$ be a maximal $2\delta R$-separated subset of $B_S(R)$.  Then the balls $B(x_j,\delta R)$ for $1 \leq j \leq m$ are disjoint and contained in $B(2R)$, so we have the upper bound
$$ m \leq \frac{|B_S(2R)|}{|B_S(\delta R)|}.$$
We then introduce the subspace 
$$ V' := \{ u \in V: u_{B(x_j,2\delta R)} = 0 \hbox{ for all } 1 \leq j \leq m; \quad u(0)=0 \}$$
where we recall that $u_B := \frac{1}{|B|} \sum_{y \in B} u(y)$ is the mean of $u$ on $B$.  Clearly $V'$ is a subspace of $V$ of codimension at most $m+1$; by \eqref{kd}, we conclude that 
\begin{equation}\label{dimv}
\dim(V') \geq k/2.
\end{equation}
Using the Gram-Schmidt process, we may then find an orthonormal (with respect to $Q_R$) basis $\tilde u_1,\ldots,\tilde u_k$ of $V$ such that $\tilde u_1, \ldots, \tilde u_{\dim(V')}$ lies in $V'$.

Let $1 \leq i \leq \dim(V')$.  From Lemma \ref{poincare} and the construction of $V'$ we have
\begin{equation}\label{bjr}
 \sum_{y \in B(x_j,2\delta R)} |\tilde u_i(y)|^2 \ll \delta^2 R^2 \sum_{y \in B(x_j,6\delta R)} |\nabla \tilde u_i(y)|^2 ).
\end{equation}
Now suppose a point $x$ is contained in $J$ balls $B(x_j,6\delta R)$.  Then the $J$ balls $B(x_j,\delta R)$ are contained in $B(x,7\delta R)$.  Since these balls are disjoint, we conclude that $J \leq |B_S(7\delta R)|/|B_S(\delta R)|$.  On the other hand, all the $B(x_j,6 \delta R)$ are contained in $B_S(2R)$.  We can therefore sum \eqref{bjr} in $j$ to conclude that
$$ \sum_{j=1}^m \sum_{y \in B(x_j,2\delta R)} |\tilde u_i(y)|^2 
\ll \delta^2 R^2 \frac{|B_S(7 \delta R)|}{|B_S(\delta R)|} \int_{y \in B_S(2R)} |\nabla \tilde u_i(y)|^2.$$
On the other hand, by construction of the $x_j$ we see that the balls $B(x_j,2\delta R)$ cover $B_S(R)$.  Since $\tilde u_i(0)=0$, we thus see from \eqref{qr} that
\begin{equation}\label{qrdel}
 Q_R(\tilde u_i) \ll \delta^2 R^2 \frac{|B_S(7 \delta R)|}{|B_S(\delta R)|} \sum_{y \in B_S(2R)} |\nabla \tilde u_i(y)|^2.
\end{equation}
To estimate this, we wish to use Lemma \ref{rev}, but first we must express $\tilde u_i$ as an approximately harmonic function.  As the $u_j$ have Lipschitz norm at most $1$, we see from \eqref{qr} that
\begin{equation}\label{qru}
 Q_R(u_j,u_j) \leq R^2 |B_S(R)|
 \end{equation}
for all $1 \leq j \leq k$.  Thus when represented in the orthonormal basis $\tilde u_1,\ldots,\tilde u_k$, the $u_1,\ldots,u_k$ are elements of $\R^k$ of norm at most $R |B_S(R)|^{1/2}$, whose coordinates form a $k \times k$ matrix whose determinant of magnitude $v$.  Taking adjugates (i.e. using Cramer's rule), we conclude that the inverse of this matrix has coefficients of size $O( v^{-1} (R |B_S(R)|^{1/2})^{k-1} )$.  In other words, the $\tilde u_1,\ldots,\tilde u_k$ are linear combinations of the $u_1,\ldots,u_k$ whose coefficients have size $O( v^{-1} (R |B_S(R)|^{1/2})^{k-1} )$.  In particular, each $\tilde u_i$ is equal to $O( k v^{-1} (R |B_S(R)|^{1/2})^{k-1} )$ times a $\eps$-harmonic Lipschitz function (note that the property of being $\eps$-harmonic Lipschitz is convex).  We may then apply Lemma \ref{rev} to conclude that
$$ \sum_{y \in B_S(2R)} |\nabla \tilde u_i(y)|^2 \ll |S|^{O(1)} (\frac{1}{R^2} \sum_{y \in B_S(4R)} |\tilde u_i(y)|^2 + k \eps^2 v^{-1} R^{k} |B_S(4R)|^{k/2}).$$
Since $\tilde u_i(0)=0$, and $\tilde u_i$ is a unit vector with respect to $Q_{4R}$, we see from \eqref{qr} that
$$\sum_{y \in B_S(4R)} |\tilde u_i(y)|^2 = 1.$$
Inserting these estimates and \eqref{vbig} into \eqref{qrdel} we obtain
$$
 Q_R(\tilde u_i) \ll |S|^{O(1)} \delta^2 \frac{|B_S(7 \delta R)|}{|B_S(\delta R)|}$$
for $1 \leq i \leq \dim(V')$.  Meanwhile, for $\dim(V') < i \leq m$ we have the crude bound
$$ Q_R(\tilde u_i) \leq Q_{4R}(\tilde u_i) = 1.$$
We conclude (using \eqref{dimv}) that
$$ \Vol_R(\tilde u_1,\ldots,\tilde u_k) \ll O(|S|)^{O(k)} \delta^{k/2}
(\frac{|B_S(7 \delta R)|}{|B_S(\delta R)|})^{k/2}.$$
On the other hand, by orthonormality we have $\Vol_{4R}(\tilde u_1,\ldots,\tilde u_k)$.  Observe that the ratio $\Vol_R(u_1,\ldots,u_k)/\Vol_{4R}(u_1,\ldots,u_k)$ is invariant under row operations, and therefore by \eqref{vdef} we have
$$ \Vol_R(u_1,\ldots,u_k) \ll O(|S|)^{O(k)} \delta^k
(\frac{|B_S(7 \delta R)|}{|B_S(\delta R)|})^k
\Vol_{4R}(u_1,\ldots,u_k)$$
and the claim follows.
\end{proof}

Now we iterate Proposition \ref{codim} to obtain

\begin{proposition}[Volume bound]\label{qkt2} Let $k \geq 1$ be an integer, $0 < \eps < 1$, $0 < \kappa < 0.1$, $d > \geq 1$, and $R_0 > 1$.  Assume that 
$$ k \geq C^{d/\kappa}$$
and
$$ R_0 \geq C^{1/\kappa} k^C$$
for some sufficiently large constant $C$.
Then for every $(R_0,d)$-growth group $(G,S)$, and every $\eps$-harmonic Lipschitz functions $u_1,\ldots,u_k: G \to \R$, we have
$$ \Vol_{R_0^{1-\kappa}}(u_1,\ldots,u_k) \ll
R_0^{O(k(d+\log |S|)) - c \frac{\kappa^2}{d+1} k \log k }
+ \eps^2 O( R_0 )^{O(kd)}
$$
for some absolute constant $c>0$.
\end{proposition}

\begin{proof}  We choose $\delta$ by the solving the equation
\begin{equation}\label{k-assum}
k = C^{\frac{d}{\kappa} \log \frac{1}{\delta}}
\end{equation}
for some large absolute constant $C$, thus
\begin{equation}\label{delta-def}
 \delta = k^{-\frac{\kappa}{d \log C}}.
\end{equation}
From our assumptions on $k,R_0$ we see (for sufficiently large choices of constants) that
\begin{equation}\label{delta-assum}
R_0^{-0.1} < \delta < 0.1.
\end{equation}
Let $n$ be the largest integer such that $4^N R_0^{1-\kappa} \leq R_0$, thus 
\begin{equation}\label{nsig}
N \gg \kappa \log R_0
\end{equation}
(note from hypothesis that $\kappa \log R_0 > 1$).  Set $R_n := 4^n R_0^{1-\kappa}$ for $1 \leq n \leq N$.  By telescoping series and \eqref{bsrd} we see that
$$ \prod_{n=1}^N \frac{|B_S(2R_n)|}{|B_S(\delta R_n)|} \ll R_0^{O( d \log \frac{1}{\delta} )} \ll O(1)^{\frac{d N}{\kappa} \log \frac{1}{\delta}}.$$
We thus conclude that for at least $N/2$ values of $1 \leq n \leq N$, we have
$$ \frac{|B_S(2R_n)|}{|B_S(\delta R_n)|} \leq O(1)^{\frac{d}{\kappa} \log \frac{1}{\delta}}.$$
By \eqref{k-assum}, we thus have (for $C$ large enough) that
$$ k \geq 2(\frac{|B_S(2R_n)|}{|B_S(\delta R_n)|} + 1)$$
for at least $N/2$ values of $n$.  For each such $n$, we may apply Proposition \ref{codim} and conclude that
$$ \Vol_{R_n}(u_1,\ldots,u_k) \leq
O(|S|)^{O(k)} \delta^{k/2}
(\frac{|B_S(7 \delta R_n)|}{|B_S(\delta R_n)|})^{k/2}
\Vol_{R_{n+1}}(u_1,\ldots,u_k)
+ k \eps^2 R_n^{k+2} |B_S(4R_n)|^{k/2}.
$$
We may make the assumption
\begin{equation}\label{vol-assum}
\Vol_{R_0^{1-\kappa}}(u_1,\ldots,u_k) > 2 k \eps^2 R_0^{k+2+dk/2}
\end{equation}
since we are done otherwise. Then (by \eqref{bsrd} and monotonicity of volume) we have
$$ \Vol_{R_n}(u_1,\ldots,u_k) >
2 k \eps^2 R_n^{k+2} |B_S(2R_n)|^{k/2}
$$
and thus
$$ \Vol_{R_n}(u_1,\ldots,u_k) \leq
O(|S|)^{O(k)} \delta^{k/2}
(\frac{|B_S(7 \delta R_n)|}{|B_S(\delta R_n)|})^{k/2}
\Vol_{R_{n+1}}(u_1,\ldots,u_k)
$$
for at least $N/2$ values of $n$.  For the other values of $n$, we see from \eqref{volmono} that
$$ \Vol_{R_n}(u_1,\ldots,u_k) \leq
\Vol_{R_{n+1}}(u_1,\ldots,u_k).
$$
Putting this all together and using monotonicity of volume again, we conclude that
$$ \Vol_{R_0^{1-\kappa}}( u_1, \ldots, u_k ) \leq O(1)^{k^2 N} |S|^{kN} \delta^{kN/4} 
(\prod_{n=0}^N \frac{|B_S(7 \delta R_n)|}{|B_S(\delta R_n)|})^{k/2}
\Vol_{R_0}(u_1,\ldots,u_k).$$
On the other hand, by telescoping series and \eqref{bsrd} we have
$$ \prod_{n=0}^N \frac{|B_S(7 \delta R_n)|}{|B_S(\delta R_n)|} \leq R_0^{O(d)}.$$
Also, arguing as in the proof \eqref{qru} we have
$$ Q_{R_0}(u_i,u_i) \leq R_0^2 |B_S(R)| \leq R_0^{d+2}$$
and thus
$$ \Vol_{R_0}(u_1,\ldots,u_k) \leq R_0^{O(kd)}.$$
We conclude that
$$ \Vol_{R_0^{1-\kappa}}( u_1, \ldots, u_k ) \leq O(|S|)^{kN} \delta^{kN/4} R_0^{O(kd)}.$$
Substituting in \eqref{delta-def}, \eqref{nsig} we obtain the claim.
\end{proof}

In practice, we will only use this proposition in the regime where $d, 1/\kappa, |S|$ are bounded, $k$ is sufficiently large depending on these parameters (but independent of $R_0$), and $\eps$ is less than extremely large negative power of $R_0$.  More precisely, we will use the following corollary of Proposition \ref{qkt2}:

\begin{corollary}[Volume bound]\label{qkt3}  Let $0 < \kappa < 0.1$, $d > 0$, $R_0>1$, and let $(G,S)$ be a $(R_0,d)$-growth group.  Suppose that 
$$ k \geq (C|S|)^{Cd^3/\kappa^2}$$
and
$$ R_0 \geq k^C$$
for some sufficiently large absolute constant $C$.  Then for any $R_0^{-Ckd}$-harmonic Lipschitz functions $u_1,\ldots,u_k: G \to \R$, we have
$$ \Vol_{R_0^{1-\kappa}}(u_1,\ldots,u_k) \leq R_0^{-100 k d}.$$
\end{corollary}

Now we can prove Theorem \ref{qkt}.  We first observe that it suffices to prove Theorem \ref{qkt} for $R_0^{-K}$-harmonic Lipschitz functions which vanish at the identity, since the general case can then be handled by adding the constant functions to $V$ (increasing the dimension by one).  By Corollary \ref{qkt3} and the hypothesis on $R_0$, we can find (if $C$ is large enough) a threshold $k_0 = O(|S|)^{O((1+d)^3/\kappa^2)}$ such that $\Vol_{R_0^{1-\kappa}}(u_1,\ldots,u_{k_0}) \leq R_0^{-100 k_0 d}$ for all $R_0^{-K}$-harmonic Lipschitz functions $u_1,\ldots,u_{k_0}: G \to \R$.  Using the greedy algorithm, one may then find $R_0^{-K}$-harmonic Lipschitz functions $u_1,\ldots,u_k: G \to \R$ for some $0 \leq k < k_0$ such that
$$ \Vol_{R_0^{1-\kappa}}(u_1,\ldots,u_k, u) \leq R_0^{-100 d} \Vol_{R_0^{1-\kappa}}(u_1,\ldots,u_k)$$
for all $R_0^{-K}$-harmonic Lipschitz functions $u: G \to \R$.  But if we let $V$ be the space spanned by $u_1,\ldots,u_k$, then from the base times height formula we have
$$ \Vol_{R_0^{1-\kappa}}(u_1,\ldots,u_k, u) = \dist_{Q_{R_0^{1-\kappa}}}( u, V ) \Vol_{R_0^{1-\kappa}}(u_1,\ldots,u_k)$$
and the claim follows.

\section{Third step for Proposition \ref{trivcor2}: Establishing an approximate isometric representation}\label{kleinersec-3}

The group $G$ acts on the space $\R^G$ of functions $u: G \to \R$ by left translation, thus
$$ \rho(g)(u)(x) := u(g^{-1} x)$$
for all $u \in \R^G$, $g \in G$, $x \in G$.  Observe that this action preserves the vector space $V$ of Lipschitz harmonic functions.  In view of Kleiner's theorem, and assuming $G$ has polynomial growth, this gives a finite-dimensional linear representation of $G$; on the quotient space $V/\R$ of $V$ modulo the constant functions, the Lipschitz semi-norm becomes a norm, which is preserved by the group action.  

The purpose of this section is to establish an analogous claim for $(R_0,d)$ groups rather than groups of polynomial growth.  It is convenient to work modulo the constants.  Let $\R^G/\R$ be the space of functions from $G$ to $\R$ modulo addition by constants, and let $\pi: \R^G \to \R^G/\R$ be the quotient map.  Observe that the action $\rho$ of $G$ on $\R^G$ descends to an action $\overline{\rho}$ on $\R^G/\R$.  One can also meaningfully define the concept of a $\eps$-harmonic Lipschitz function $\overline{u}$ in $\R^G/\R$, since this concept is invariant under addition by constants.  We can also define induced $\overline{\ell^p(B)}$ norms for every finite $B \subset G$ by
$$ \| \overline{u} \|_{\overline{\ell^p(B)}} := \inf \{ \|u\|_{\ell^p(B)}: \pi(u) = \overline{u} \}$$
for any $\overline{u} \in \R^G/\R$.  For $p=2$, this norm is associated with an inner product
$$ \langle \overline{u}, \overline{v} \rangle_{\overline{\ell^2(B)}} = \sum_{y \in B} u(y) v(y)$$
where $u,v$ are the unique lifts of $\overline{u}, \overline{v}$ by $\pi$ that have mean zero on $B$.

We will need a quotiented variant of Theorem \ref{qkt} with an additional ``good scale'' $R_1$ which is stable under translations.

\begin{proposition}[Quantitative Kleiner theorem with good scale]\label{bkts}
Let $0 < \kappa < 0.1$, and let $(G,S)$ be a $(R_0,d)$-growth group with $d \geq 1$ obeying 
\begin{equation}\label{rbit2}
 R_0 \geq K := (C|S|)^{C d^3/\kappa^3}
\end{equation}
for some sufficiently large $C$.  Then there exists a finite-dimensional subspace $\overline{V}$ of $\R^G/\R$ of dimension
$$ \dim(\overline{V}) \leq O(|S|)^{O(d^3/\kappa^2)}$$
and a scale $R_0^{1-2\kappa} \leq R_1 \leq R_0^{1-\kappa}$ with two properties:
\begin{itemize}
\item Every $R_0^{-K}$-harmonic Lipschitz function $\overline{u} \in \R^G/\R$ lies at a distance at most $R_0^{-99d}$ from $\overline{V}$ in the $\overline{\ell^2(B_S(R_1))}$ norm.
\item For every $\overline{u} \in \overline{V}$, we have
\begin{equation}\label{uinv}
 \| \overline{u} \|_{\ell^2(B_S(R_1))} \leq (1 + R_0^{-\kappa}) \| \overline{u} \|_{\ell^2(B_S(R_1-R_0^{1-4\kappa}))}. 
\end{equation}
\end{itemize}
\end{proposition}

\begin{proof}   We first make the observation that we may replace \eqref{uinv} by the variant condition
\begin{equation}\label{uinv-2}
\| \overline{u} \|_{\overline{\ell^2(B_S(R_1/2))}} \geq R_0^{-200d} \| \overline{u} \|_{\overline{\ell^2(B_S(R_1))}}
\end{equation}
at the slight cost of tightening the lower bound $R_1 \geq R_0^{1-2\kappa}$ to $R_1 \geq 2 R_0^{1-2\kappa}$.  Indeed, suppose that \eqref{uinv-2} held for all $\overline{u} \in \overline{V}$.  Then, if we let $Q_r: V \to \R$ denote the positive definite quadratic form $Q_r(\overline{u}) := \| \overline{u} \|_{\overline{\ell^2(B_S(r))}}^2$ for each $r>0$, we see that $\det Q_{R_1} / \det Q_{R_1/2} \leq (R_0^{200d})^{2\dim(\overline{V})}$ (where we pick some arbitrary fixed basis of $V$ with which to compute determinants).  As the $Q_r$ are increasing in $r$, we may then use the pigeonhole principle (and the bounds on $R_0$ and $\dim(\overline{V})$) and find $R_1/2 \leq R'_1 \leq R_1$ such that
$$ \det Q_{R'_1} \leq (1 + R_0^{-\kappa})^2 \det Q_{R'_1 - R_0^{1-4\kappa}};$$
diagonalizing the quadratic form $Q_{R'_1}$ with respect to $Q_{R'_1  - R_0^{1-4\kappa}}$ we obtain \eqref{uinv} with $R_1$ replaced by the slightly smaller $R'_1$.

It remains to find a $\overline{V}$ and $R_1$ obeying \eqref{uinv-2}, as well as the property about $\overline{V}$ approximating $R_0^{-K}$-harmonic Lipschitz functions. From Theorem \ref{qkt} and a quotienting by $\pi$, we see that if we (temporarily) set $R_1 := R_0^{1-\kappa}$, we may find a subspace $\overline{V}_0$ of $\R^G/\R$ of dimension $O(|S|)^{O(d^3/\kappa^2)}$ such that every $R_0^{-K}$-harmonic Lipschitz function $\overline{u} \in \R^G/\R$ lies at a distance at most $R_0^{-100d}$ from $\overline{V}_0$ in the $\overline{\ell^2(B_S(R_1))}$ norm.

If the property \eqref{uinv-2} holds for all $\overline{u} \in \overline{V}_0$ with this choice of $R_1$, then we are done.  Otherwise, suppose there exists $\overline{u}_0 \in \overline{V}_0$ for which 
\begin{equation}\label{unorm}
 \| \overline{u}_0 \|_{\overline{\ell^2(B_S(R_1/2))}} < R_0^{-200d} \| \overline{u}_0 \|_{\overline{\ell^2(B_S(R_1))}}
\end{equation}
Clearly, $\overline{u}_0$ is non-zero; we may normalize 
\begin{equation}\label{unorm2}
\| \overline{u}_0 \|_{\overline{\ell^2(B_S(R_1))}}=1.
\end{equation}
Let $\overline{V}_1$ be the orthogonal complement of $\overline{u}_0$ in $\overline{V}_0$, thus $\overline{V}_1$ has dimension one less than $\overline{V}_0$.

Now let $\overline{u} \in \R^G/\R$ be a $R_0^{-K}$-harmonic Lipschitz function.  By construction, we can find $\overline{v} \in \overline{V}_0$ such that 
\begin{equation}\label{uv}
 \| \overline{u} - \overline{v} \|_{\overline{\ell^2(B_S(R_1))}} \leq R_0^{-100d}.
\end{equation}
On the other hand, $\overline{u}$ has Lipschitz constant at most $1$ and thus (since constants have been quotiented out)
$$ \| \overline{u} \|_{\overline{\ell^2(B_S(R_1))}} \leq R_1 |B_S(R_1)|^{1/2} \leq R_0^{1+d/2}.$$
In particular
$$ \| \overline{v} \|_{\overline{\ell^2(B_S(R_1))}} \ll R_1^{1+d/2}$$
and hence by Cauchy-Schwarz and \eqref{unorm2}
\begin{equation}\label{cs}
 \langle \overline{v}, \overline{u}_0 \rangle_{\overline{\ell^2(B_S(R_1))}} = O(R_0^{1+d/2}).
\end{equation}

We split $\overline{v} = \overline{v}_1 + \langle \overline{v}, \overline{u}_0 \rangle_{\overline{\ell^2(B_S(R_1))}} \overline{u_0}$, where $\overline{v}_1 \in \overline{V}_1$ is the orthogonal projection of $\overline{v}$ to $\overline{V}_1$.  From \eqref{unorm}, \eqref{unorm2}, \eqref{cs} the latter term is small on $B_S(R_1/2)$:
$$ \| \langle \overline{v}, \overline{u}_0 \rangle_{\overline{\ell^2(B_S(R_1))}} \overline{u_0} \|_{\overline{\ell^2(B_S(R_1/2))}} \leq R_0^{-150d}$$
(say).  From this, \eqref{uv}, and the triangle inequality we conclude that
$$
 \| \overline{u} - \overline{v}_1 \|_{\overline{\ell^2(B_S(R_1/2))}} \leq R_0^{-100d} + R_0^{-150d}.
$$
Thus, the property that $\overline{V}_0$ approximates $R_0^{-K}$-harmonic Lipschitz functions has been inherited by $\overline{V}_1$, at the slight cost of reducing the scale $R_1$ to $R_1/2$ and increasing the error of approximation slightly from $R_0^{-100d}$ to $R_0^{-100d} + R_0^{-150d}$.  One can then iterate this process at most $\dim(\overline{V}_0) = O(|S|)^{O(d^3/\kappa^2)}$ times until we find a space $\overline{V}$ and a scale $R_1$ for which \eqref{uinv-2} is satisfied; the dimension bound and the largeness hypothesis on $R_0$ ensures that $R_1 \geq R_0^{1-2\kappa}$, and that the total error of approximation never exceeds $R_0^{-99d}$.  The claim follows.
\end{proof}

Let $\kappa, (G,S), R_0, d, K, \overline{V}, R_1$ be as in Proposition \ref{bkts}, and set $D := \dim(\overline{V})$, thus
\begin{equation}\label{dimbound}
D = O(|S|)^{O(d^3/\kappa^2)}.
\end{equation}
In practice, $D$ should be viewed as bounded (especially when compared with the large parameter $R_0$), and so the factors of $D^{O(1)}$ that appear below should be ignored at a first reading.

Let $\Omega \subset \overline{V}$ be the set of all elements of $\overline{V}$ which lie within a distance $R_0^{-99d}$ in $\overline{\ell^2(B_S(R_1))}$ norm from a $R_0^{-K}$-harmonic Lipschitz function.  This is a symmetric convex subset of $\overline{V}$ with non-empty interior.  Applying John's theorem \cite{john}, we may then find an ellipsoid $E \subset \overline{V}$ such that $E \subset \Omega \subset \sqrt{D} \cdot E$.  If we let $e_1,\ldots,e_D \in \overline{V}$ and $\lambda_1,\ldots,\lambda_D > 0$ be the principal orthonormal directions and radii of this ellipsoid with respect to the Hilbert space structure $\overline{\ell^2(B_S(R_1))}$ on $\overline{V}$, we thus see that
$$ \lambda_i e_i \in \Omega$$
for $i=1,\ldots,D$, and conversely every element of $\Omega$ can be represented in the form
\begin{equation}\label{repr}
 \sum_{i=1}^D t_i \lambda_i e_i
\end{equation}
for some $t_1,\ldots,t_D = O(D^{O(1)})$.

Since $\Omega$ (and hence $\sqrt{D} \cdot E$) contains the ball of radius $R_0^{-99d}$, we have the lower bound
\begin{equation}\label{lam-lower}
\lambda_i \geq R_0^{-99d} / \sqrt{D}
\end{equation}
for all $1 \leq i \leq D$.  Also since $R_0^{-K}$-harmonic Lipschitz functions have an $\overline{\ell^2(B_S(R_1))}$ norm of at most $R_1 |B_S(R_1)|^{1/2} \leq R_0^{(d+1)/2}$, we have the upper bound
\begin{equation}\label{lam-upper}
\lambda_i \ll R_0^{(d+1)/2}
\end{equation}
for $1 \leq i \leq D$.

Let $1 \leq i \leq D$ and $g \in G$.  By construction, there exists a $R_0^{-K}$-harmonic Lipschitz function $\overline{u}_i \in \R^G/\R$ with $\| \overline{u_i} - \lambda_i e_i \|_{\ell^2(B_S(R_1))} \leq R_0^{-99d}$.  Translating this, we obtain
$$ \| \overline{\rho}(g) \overline{u_i} - \lambda_i \overline{\rho}(g) e_i \|_{\ell^2(g \cdot B_S(R_1))} \leq R_0^{-99d}.$$
In particular, if $g \in B_S(R_0^{1-5\kappa})$, then
$$ \| \overline{\rho}(g) \overline{u_i} - \lambda_i \overline{\rho}(g) e_i \|_{\ell^2(B_S(R_1-R_0^{1-5\kappa}))} \leq R_0^{-99d}.$$
On the other hand, $\overline{\rho}(g) \overline{u_i}$ is also a $R_0^{-K}$-harmonic Lipschitz function, and thus must lie within $R_0^{-99d}$ in $\ell^2(B_S(R_1))$ norm of some function $f_{g,i} \in \Omega$.  By the triangle inequality we thus have
\begin{equation}\label{tgh}
 \| f_{g,i} - \lambda_i \overline{\rho}(g) e_i \|_{\ell^2(B_S(R_1-R_0^{1-5\kappa}))} \leq 2 R_0^{-99d}.
\end{equation}
Using \eqref{repr}, we can write
\begin{equation}\label{tgh-2}
f_{g,i} = \sum_{j=1}^D t_{g,j,i} \lambda_j e_j
\end{equation}
for some $t_{g,i,j}$ obeying the bounds
\begin{equation}\label{tgh-bound0}
|t_{g,j,i}| \ll D^{O(1)}.
\end{equation}
Another bound on these coefficients is obtained by observing that
$$ \| \lambda_i \overline{\rho}(g) e_i \|_{\ell^2(B_S(R_1-R_0^{1-5\kappa}))} \leq \lambda_i \| e_i \|_{\ell^2(B_S(R_1))} = \lambda_i$$
and thus (by \eqref{tgh}, \eqref{lam-lower})
$$ \|f_{g,i} \|_{\ell^2(B_S(R_1-R_0^{1-5\kappa}))} \ll D^{O(1)} \lambda_i$$
and thus (by \eqref{uinv})
$$ \|f_{g,i} \|_{\ell^2(B_S(R_1))} \ll D^{O(1)} \lambda_i.$$
Using \eqref{tgh-2} and the orthonormal properties of $e_j$ we conclude that
\begin{equation}\label{tgh-bound}
|t_{g,j,i}| \ll D^{O(1)} \lambda_i / \lambda_j.
\end{equation}

We now investigate the extent to which the $D \times D$ matrices $U_g := (t_{g,j,i})_{1 \leq j,i \leq D}$ behave like a representation.  From construction we see that we may take $U_\id = I$, where $I$ is the $D \times D$ matrix.  Now we look at the multiplicativity.

\begin{proposition}[$U_g$ approximately multiplicative]\label{perp}  
If $g, h \in B_S(R_0^{1-5\kappa}/2)$ and $1 \leq i,k \leq D$, then the $(k,i)$ entry of the matrix $U_{gh} - U_g U_h$, i.e.
$$ t_{gh,k,i} - \sum_{j=1}^d t_{g,k,j} t_{h,j,i},$$
has magnitude $O( D^{O(1)} R_0^{-99d} / \lambda_k )$.
\end{proposition}

\begin{proof} From \eqref{tgh}, \eqref{tgh-2} we have
$$ \| \sum_{j=1}^D t_{h,j,i} \lambda_j e_j - \lambda_i \overline{\rho}(h) e_i \|_{\ell^2(B_S(R_1-R_0^{1-5\kappa}))} \leq 2 R_0^{-99d}$$
for all $1 \leq i \leq D$; applying $\overline{\rho}(g)$ to this, we conclude
$$ \| \sum_{j=1}^D t_{h,j,i} \lambda_j \overline{\rho}(g) e_j - \lambda_i \overline{\rho}(gh) e_i \|_{\ell^2(B_S(R_1-2 R_0^{1-5\kappa}))} \leq 2 R_0^{-99d}.$$
Meanwhile, from \eqref{tgh}, \eqref{tgh-2} we have
$$ \| \lambda_j \overline{\rho}(g) e_j - \sum_{k=1}^D t_{g,k,j} \lambda_k e_k \|_{\ell^2(B_S(R_1-2 R_0^{1-5\kappa}))} \leq 2 R_0^{-99d}
$$
and thus by the triangle inequality we have
$$ \| \lambda_i \overline{\rho}(gh) e_i - \sum_{k=1}^D \sum_{j=1}^D t_{h,j,i} t_{g,k,j} \lambda_k e_k \|_{\ell^2(B_S(R_1-2 R_0^{1-5\kappa}))} \ll D^{O(1)} R_0^{-99d}.
$$
Meanwhile, from one final application of \eqref{tgh}, \eqref{tgh-2} we have
$$ \| \lambda_i \overline{\rho}(gh) e_i - \sum_{k=1}^D t_{gh,k,i} \lambda_k e_k \|_{\ell^2(B_S(R_1-2 R_0^{1-5\kappa}))} \ll D^{O(1)} R_0^{-99d}
$$
and so by the triangle inequality
$$ \| \sum_{k=1}^D [t_{gh,k,i} - \sum_{j=1}^d t_{h,j,i} t_{g,k,j}] \lambda_k e_k \|_{\ell^2(B_S(R_1-2 R_0^{1-5\kappa}))} \ll D^{O(1)} R_0^{-99d}
$$
and thus by \eqref{uinv}
$$ \| \sum_{k=1}^D [t_{gh,i,k} - \sum_{j=1}^d t_{h,i,j} t_{g,j,k}] \lambda_k e_k \|_{\ell^2(B_S(R_1))} \ll D^{O(1)} R_0^{-99d}
$$
As the $e_k$ are orthonormal, the claim follows.
\end{proof}

\section{Fourth step for Proposition \ref{trivcor2}: Taking commutators}\label{kleinersec-4}

We continue the discussion in the previous section.  To simplify the expressions slightly we will make the smallness assumption
\begin{equation}\label{kappa-small}
\kappa \leq \frac{1}{d}, \frac{1}{\log |S|}
\end{equation}
on $\kappa$, and in particular from \eqref{dimbound}
\begin{equation}\label{dimbound-2}
D \leq 2^{O(\kappa^{-O(1)})}.
\end{equation}
For reasons that will be clearer later, we will also need to make $R_0$ larger than previously assumed, in particular we assume that
\begin{equation}\label{r0-big}
 R_0 \geq 2^{2^{C/\kappa^C}}
\end{equation}
for some sufficiently large absolute constant $C$.  (In particular, quantities such as $R_0^{\kappa^{10}/2D^2}$ are still quite large.)

The next step is to locate a large set of group elements $g \in G$ (which will be commutators of other group elements) for which $U_g$ are very close to the identity matrix $I$.  The key point here is that if $U_g, U_h$ are within $\eps$ of $I$ for some small $\eps > 0$ (in some suitable matrix norm), then the commutator $[U_g,U_h]$ is within $O(|D|^{O(1)} \eps^2)$ of $I$.  Meanwhile, from \eqref{perp}, we expect $[U_g,U_h] \approx U_{[g,h]}$.  The strategy here can be viewed as a simplified variant of the argument used to prove the Solovay-Kitaev theorem \cite{sk}.

We turn to the details.  Let us write ${\bigO}_\lambda(X)$ to denote any matrix whose $(k,i)$ entry is $O(X/\lambda_k)$, and ${\bigO}_\lambda^\lambda(X)$ to denote any matrix whose $(k,i)$ entry is $O(X \min(1, \lambda_i/\lambda_k ))$. Then we can rewrite the conclusion of Proposition \ref{perp} as
\begin{equation}\label{ugh}
 U_{gh} = U_g U_h + {\bigO}_\lambda( D^{O(1)} R_0^{-99d} )
\end{equation}
and rewrite \eqref{tgh-bound}, \eqref{tgh-bound0} as
\begin{equation}\label{crude}
U_g = {\bigO}_\lambda^\lambda( D^{O(1)} ).
\end{equation}
We also observe the multiplication laws
\begin{equation}\label{mult}
{\bigO}_\lambda^\lambda(X) {\bigO}_\lambda^\lambda(Y) = {\bigO}_\lambda^\lambda( D^{O(1)} XY ); \quad {\bigO}_\lambda^\lambda(X) {\bigO}_\lambda(Y), {\bigO}_\lambda(X) {\bigO}_\lambda^\lambda(Y) = {\bigO}_\lambda( D^{O(1)} XY ).
\end{equation}
In particular we have
\begin{equation}\label{guo}
 U_g U_{g^{-1}}, U_{g^{-1}} U_g = I + {\bigO}_\lambda( D^{O(1)} R_0^{-99d} )
\end{equation}
for all $g \in B_S(R_0^{1-5\kappa}/2)$.

We have the following fundamental fact:

\begin{lemma}[Commutator bound]\label{comb}  If $g, e, e' \in B_S(R_0^{1-5\kappa}/100)$ and $0 < \eps \leq 1$ are such that 
\begin{equation}\label{boss}
U_e, U_{e'} = I + {\bigO}_\lambda^\lambda(\eps) + {\bigO}_\lambda( X ),
\end{equation}
for some $R_0^{-99d} \leq X \leq R_0^{-98d}$, then
$$U_{g[e,e']g^{-1}} = I + {\bigO}_\lambda^\lambda( D^{O(1)} \eps^2 ) + {\bigO}_\lambda( D^{O(1)} X ).$$
\end{lemma}

\begin{proof}  From \eqref{ugh}, \eqref{crude}, \eqref{mult} we have 
$$
U_{[e,e']} = U_e U_{e'} U_{e^{-1}} U_{(e')^{-1}} + {\bigO}_\lambda( D^{O(1)} X ).
$$
Splitting $U_{e'} = (U_{e'}-I) + I$ and using \eqref{guo}, \eqref{crude}, \eqref{mult} we conclude that
\begin{equation}\label{uho}
U_{[e,e']} = U_e (U_{e'}-I) U_{e^{-1}} U_{(e')^{-1}} + U_{(e')^{-1}} + {\bigO}_\lambda( D^{O(1)} X ).
\end{equation}
From \eqref{boss}, \eqref{mult} we have
$$ (U_e-I) (U_{e'}-I), (U_{e'}-I) (U_e-I) = {\bigO}_\lambda^\lambda(D^{O(1)} \eps^2) + {\bigO}_\lambda(D^{O(1)} X)$$
and thus
$$ U_e (U_{e'}-I) = (U_{e'}-I) U_e + {\bigO}_\lambda^\lambda(D^{O(1)} \eps^2) + {\bigO}_\lambda(D^{O(1)} X).$$
Inserting this into \eqref{uho} and using \eqref{ugh}, \eqref{crude}, \eqref{mult} we conclude that
$$
U_{[e,e']} = (U_{e'}-I) U_{(e')^{-1}} + U_{(e')^{-1}} + {\bigO}_\lambda^\lambda(D^{O(1)} \eps^2) + {\bigO}_\lambda( D^{O(1)} X ).
$$
Applying \eqref{uho} again we obtain 
$$U_{[e,e']} = I + {\bigO}_\lambda^\lambda( D^{O(1)} \eps^2 ) + {\bigO}_\lambda( D^{O(1)} X ).$$
If we multiply this on the left by $U_g$ and on the right by $U_{g^{-1}}$ and use \eqref{guo}, \eqref{crude}, \eqref{mult} we obtain the claim.
\end{proof}

We would like to iterate this bound to find many $g$ with $U_g$ very close (e.g. $O(R_0^{-50d})$) to $I$, but to get started we will need to locate a preliminary supply of $g$ for which $U_g$ is somewhat close (e.g. $O(R_0^{-\kappa/2D^2})$) to $I$.  Morally, this should follow from the Dirichlet box principle (i.e. the pigeonhole principle) since the $U_g$ (and $U_{g^{-1}}$) are morally localized to a compact set of matrices (thanks to \eqref{crude}) and are approximately multiplicative.  We now make this intuition precise.

\begin{lemma}[Box principle] There exists a subgroup $(G',S')$ of $G$ with $S' \subset B_S(R_0^{\kappa^{10}})$ and finite index $|G:G'| \leq R_0^{\kappa^{10}}$ such that $U_e = I + {\bigO}_\lambda^\lambda( D^{O(1} R_0^{-\kappa^{10}/2D^2} ) + {\bigO}_\lambda(D^{O(1)} R_0^{-99d} )$ for all $e \in S'$.
\end{lemma}

\begin{proof}  From \eqref{crude}, the matrices $U_g$ for $g \in B_S(R_0^{1-5\kappa}/2)$ are contained in a set of matrices the form $\{ U: U = {\bigO}_\lambda^\lambda( D^{O(1)} ) \}$.  From the convexity of the conditions used in the ${\bigO}_\lambda^\lambda()$ notation (and the largeness hypothesis on $R_0$), we may cover this $D^2$-dimensional set by $M \leq R_0^{\kappa^{10}}$ balls $B_1,\ldots,B_M$ of the form $B_m = \{ U: U = U_m + {\bigO}_\lambda^\lambda( R^{-\kappa^{10}/2D^2} ) \}$ for some $U_m =  {\bigO}_\lambda^\lambda( D^{O(1)} )$.  (Here we are using \eqref{r0-big} to clean up the bounds somewhat.)

For each $r > 0$, let $A_r \subset \{1,\ldots,M\}$ be the set of those $1 \leq m \leq M$ such that $U_g \in B_m$ for some $g \in B_S(r)$.  Clearly the $A_r$ are increasing in $r$, so by the pigeonhole principle there exists $1 \leq r \leq M$ such that $A_{r+1} = A_r$.   

Fix this $r$.  For each $m \in A_r$, let $g_m \in B_S(r)$ be a representative such that $U_{g_m} \in B_m$.  Since $A_{r+1}=A_r$, we see that for each $g \in B_S(r+1)$ there exists $m \in A_r$ such that $U_g \in B_m$, and in particular
$$ U_g = U_{g_m} + {\bigO}_\lambda^\lambda( R_0^{-\kappa^{10}/2D^2} ).$$
Multiplying by $U_{g_m^{-1}}$ on the right and using \eqref{guo}, \eqref{crude}, \eqref{mult}, we see that
$$ U_{g g_m^{-1}} = I + {\bigO}_\lambda^\lambda( D^{O(1)} R_0^{-\kappa^{10}/2D^2} ) + {\bigO}_\lambda( D^{O(1)} R_0^{-99d} )$$
and similarly
$$ U_{g_m g^{-1}} = I + {\bigO}_\lambda^\lambda( D^{O(1)} R_0^{-\kappa^{10}/2D^2} ) + {\bigO}_\lambda( D^{O(1)} R_0^{-99d} ).$$
Let $S'$ denote the set of all $g g_m^{-1}, g_m g^{-1}$ that arise in this manner, then $S' \subset B_S(2r+1)$ is symmetric and
$$ B_S(r+1) \subset S' \cdot \{g_1,\ldots,g_M\} \subset S' \cdot B_S(r).$$
Iterating this we see that
$$ B_S(r+n) \subset B_{S'}(n) \cdot B_S(r) \subset B_{S'}(n+1) \cdot \{g_1,\ldots,g_M\}$$
for all $n$; thus if $G'$ denotes the group generated by $S'$, then on taking unions as $n \to \infty$ we conclude that
$$ G \subset G' \cdot \{g_1,\ldots,g_m\}.$$
Thus $G'$ has index at most $M$, and the claim follows.
\end{proof}

Write $\eps := R_0^{-\kappa^{10}/2D^2}$, thus 
\begin{equation}\label{emm}
U_e = I + {\bigO}_\lambda^\lambda(D^{O(1)} \eps) + {\bigO}_\lambda(D^{O(1)} R_0^{-99d})
\end{equation}
 for all $e \in S'$.

Since $(G,S)$ is a $(R_0,d)$-growth group and $S' \subset B_S(R_0^{\kappa^{10}})$, we see that $(G',S')$ is a $(R_0^{1-\kappa^{10}},d/(1-\kappa^{10}))$-growth group.  Applying Lemma \ref{zonk}, we see that the commutator group $(G')^{(2)} := [G',G']$ can be generated by a set $(S')^{(2)}$ of generators in $B_{S'}(R_0^{\kappa^{10}})$ of the form $g[e,e']g^{-1}$ for some $e, e' \in S'$ and $g \in B_{S'}(R_0^{\kappa^{10}})$, and furthermore $((G')^2,(S')^{(2)})$ is a $(R_0^{1-2\kappa^{10}},d/(1-2\kappa^{10}))$-growth group.  From Lemma \ref{comb} we conclude that
$$ U_e = I + {\bigO}_\lambda^\lambda(D^{O(1)} \eps^2) + {\bigO}_\lambda(D^{O(1)} R_0^{-99d})$$
for all $e \in (S')^{(2)}$, where the $O(1)$ exponents are larger than those in \eqref{emm} by a multiplicative absolute constant.

Let $l$ be the first integer such that
\begin{equation}\label{epsd}
\eps^{2^l} < R_0^{-100d}
\end{equation}
or equivalently
$$2^l (\kappa/2D^2) > 100 d$$
thus by \eqref{dimbound}, \eqref{kappa-small}
\begin{equation}\label{lbound}
 l \ll \frac{d^3}{\kappa^2} (1 + \log |S|) \ll \kappa^{-6}.
 \end{equation}
We can iterate the above procedure $l$ times and conclude that the $l^{th}$ group $(G')^{(l)}$ in the derived series of $G'$ is generated by a set $(S')^{(l)} \in B_{S'}(R_0^\kappa)$ with the property that
$$ U_e = I + {\bigO}_\lambda^\lambda(D^{2^{O(l)}} \eps^{2^l}) + {\bigO}_\lambda(D^{2^{O(l)}} R_0^{-99d})$$
for all $e \in (S')^{(l)}$.
By \eqref{r0-big}, \eqref{dimbound-2}, \eqref{lam-upper}, \eqref{epsd} we may clean this up as
$$ U_e = I + {\bigO}_\lambda( R_0^{-90 d}) $$
(say). From \eqref{tgh}, \eqref{tgh-2} we conclude that
$$ \| \rho(e) (\lambda_i e_i) - \lambda_i e_i \|_{\overline{\ell^2(B_S(R_1))}} \ll R_0^{-80 d}$$
(say) for all $1 \leq i \leq D$ and $e \in (S')^{(l)}$, which by \eqref{repr} implies that
$$ \| \rho(e) f - f \|_{\overline{\ell^2(B_S(R_1))}} \ll R_0^{-70d}$$
(say) for all $f \in \Omega$ and $e \in (S')^{(l)}$.  In particular, for any $R_0^{-K}$-harmonic Lipschitz function $u$, we see from the triangle inequality that
$$ \| \rho(e) u - u \|_{\overline{\ell^2(B_S(R_1))}} \ll R_0^{-60d}$$
(say) for all $e \in (S')^{(l)}$.  Unpacking the definition of the $\overline{\ell^2(B_S(R_1))}$ norm, and noting that $B_S(R_1)$ has cardinality at most $R_0^d$, we conclude that
$$ u(e g) - u(e h) = u(g) - u(h) + O( R_0^{-50d})$$
(say) for all $g, h \in B_S(R_1)$ and $e \in (S')^{(l)}$ (here we use the symmetry of $(S')^{(l)}$).  In particular, for $e,e' \in (S')^{(l)}$ and $g \in B_S(R_1/2)$ we have
$$ u(e e' g) = u(e g) + u( e' g ) - u( g ) + O( R_0^{-50d}).$$
Reversing $e$ and $e'$ and subtracting we conclude that
$$ u(e e' g) = u(e' e g) + O( R_0^{-50d}),$$
and thus
$$ u([e, e'] g) = u(g) + O( R_0^{-50d}),$$
for $e,e' \in (S')^{(l)}$ and $g \in B_S(R_1/4)$, which implies (after replacing $g$ by $g^{-1} h$, and $u(\cdot)$ by $u(g \cdot)$) that
\begin{equation}\label{ughr}
u(g [e, e'] g^{-1} h) = u(h) + O( R_0^{-50d})
\end{equation}
for all $g, h \in B_S(R_1/8)$ and $e,e' \in (S')^{(l)}$.

By invoking Lemma \ref{zonk} one last time, one can find a set $(S')^{(l+1)}$ of generators of $(G')^{(l+1)}$ in $B_{S'}(R_0^{2\kappa})$ of the form $g[e,e']g^{-1}$ for some $e,e' \in (S')^{(l)}$ and $g \in B_S(R_0^{2\kappa})$, and so from \eqref{ughr} we have concluded the following:

\begin{theorem}[Many trivial directions for harmonic Lipschitz functions]\label{trivdir}  Let $0 < \kappa < 0.1$, $R_0, d \geq 1$, and let $G$ be a $(R_0,d)$-growth group.  Assume the bounds \eqref{kappa-small}, \eqref{r0-big} for some sufficiently large $C$.  Then there exists a $(R_0^{1-\kappa^{10}},d/(1-\kappa^{10}))$-growth subgroup $(G',S')$ of $G$ of index at most $R_0^{\kappa^{10}}$ and a positive integer $l = O(\kappa^{-6})$ such that $(G')^{(l+1)}$ is generated by a set $(S')^{(l+1)} \subset B_{S'}(R_0^{2\kappa})$ obeying the bound
\begin{equation}\label{ughr-2}
u(ex) = u(x) + O( R_0^{-50d})
\end{equation}
for all $e \in (S')^{(l+1)}$, $x \in B_S(R_0^{1-3\kappa})$, and all $R_0^{-K}$-harmonic Lipschitz functions $u: G \to \R$, where $K$ is defined by \eqref{rbit2}.
\end{theorem}

\begin{remark} An instructive example here is that of the Heisenberg group $G = \begin{pmatrix} 1 & \Z & \Z \\ 0 & 1 & \Z \\ 0 & 0 & 1 \end{pmatrix}$, with the elementary row operations as generators.  One can show that the only harmonic Lipschitz functions $u: G \to \R$ are those functions which are affine-linear combinations of the near-diagonal coefficients $x, z$ of the group element $\begin{pmatrix} 1 & x & y \\ 0 & 1 & z \\ 0 & 0 & 1 \end{pmatrix}$, and in particular such functions are invariant with respect to the vertical element $e := \begin{pmatrix} 1 & 0 & 1 \\ 0 & 1 & 0 \\ 0 & 0 & 1 \end{pmatrix}$.  One way to see this is to use the harmonicity to observe the reproducing formula $u = u * \sigma^{(m)}$ for any $m$, where $\sigma^{(m)}$ are the random walk distributions defined in Proposition \ref{infilip}.  As $e$ is central, one then has $\partial_e u = u * \partial_e \sigma^{(m)}$, where $\partial_e f(x) := f(ex) - f(x)$.  But a computation shows that the total variation of $\partial_e \sigma^{(m)}$ is $O(1/m^2)$ (the intuition here is that $\sigma^{(m)}$ behaves like uniform probability measure on a box of dimensions $O(\sqrt{m})$ in the $x,z$ directions and $O(m)$ in the $y$ directions); in contrast, $u$, being Lipschitz, can only fluctuate by at most $O(m)$ on the bulk of the support of $\sigma^{(m)}$.  Estimating things carefully and taking limits as $m \to \infty$ we conclude that $\partial_e u = 0$, at which point it is easy to verify the claim.  This example illustrates 
the general phenomenon, established in our companion paper~\cite{Sh-Ta} (using a different method),
 that Lipschitz harmonic functions on nilpotent groups vanish along iterated 
commutator directions; in fact modulo the constants they are exactly the 
additive group characters. 
\end{remark}

\section{Final step for Proposition \ref{trivcor2}: non-trivial harmonic Lipschitz functions have large range}\label{kleinersec-5}

An easy application of the maximum principle shows that any non-constant harmonic function $u: G \to \R$ must attain an infinite number of values.  We thus expect any $\eps$-harmonic Lipschitz function $u: G \to \R$ obeying some non-degeneracy condition (e.g. a lower bound on $\nabla u(\id)$) to also take on a large number of values in any given ball $B_S(R)$.

In fact we will need a stronger result (under a polynomial growth hypothesis), which asserts that a non-degenerate $\eps$-harmonic function must in fact fluctuate by $\gg R$ on the ball $B_S(R)$:

\begin{proposition}[Lower bound on range]\label{lorrange}  Let $(G,S)$ be a $(R_0,d)$-growth group for some $R_0, d \geq 1$, let $0 < \kappa < 0.1$, and suppose that \eqref{rbit} holds for some sufficiently large absolute constant $C$. Let $u:G \to \R$ be an $R_0^{-K}$-harmonic function such that $|\nabla u(\id)| \geq 1/|S|$, where $K$ was defined in \eqref{rbit}.  Then for every $R_0^{1-5\kappa} \leq R \leq R_0^{1-2\kappa}$ we have
\begin{equation}\label{super}
\sup_{x \in B_S(R)} |u(x)-u(\id)| \gg O(|S|)^{-O(d^3/\kappa^2)} R_0^{-\kappa} R.
\end{equation}
\end{proposition}

\begin{proof}  Our main tool here will be the quantitative Kleiner theorem (Theorem \ref{qkt}). This theorem gives us a space $V \subset \R^G$ of dimension
$$ D:= \dim(V) = O(|S|)^{O(d^3/\kappa^2)}$$
such that for every $g \in G$, there exists $v_g \in V$ such that
$$ \| \rho(g) u- v_g \|_{\ell^2(B_S(R_0^{1-\kappa}))} \leq R_0^{-100d}$$
and in particular
\begin{equation}\label{Rhogio}
 \| \rho(g) \nabla u - \nabla v_g \|_{\ell^2(B_S(R_0^{1-\kappa}))} \leq |S| R_0^{-100d}.
\end{equation}

The balls $B_S(R')$ for $R_0^{-\kappa} R \leq R' \leq R$ are increasing in $R'$, and have cardinality between $1$ and $R_0^d$.  By the pigeonhole principle, one may thus find a radius $2R_0^{-\kappa} R \leq R' \leq R$ such that
\begin{equation}\label{nob}
|B_S(R')| \ll O(1)^{d/\kappa} |B_S(R'/10)|.
\end{equation}

Fix this $R'$.  By subtracting off a constant, we may assume $u(\id)=0$.  Write $\alpha := R' R_0^{-K} + R^{-1} \sup_{x \in B_S(R')} |u(x)|$, then
$$ \sum_{x \in B_S(R')} |u(x)|^2 \leq (R')^2 \alpha^2 |B_S(R')|.$$
Applying Proposition \ref{rev} we conclude that
\begin{equation}\label{nab}
 \| \nabla u \|_{\ell^2(B_S(R'/2))} \ll |S|^{O(1)} \alpha |B_S(R')|^{1/2}.
\end{equation}

Now, let $g_1,\ldots,g_{D+1}$ be chosen uniformly at random from $B_S(R'/10)$.  From \eqref{nob}, \eqref{nab}, and Chebyshev's inequality we see that for each distinct $1 \leq i,j \leq D+1$, we have
\begin{equation}\label{rhogi}
|\rho(g_i) \nabla(u)(g_j)| \leq \frac{1}{100 |S| D}
\end{equation}
with probability at least $1 - O(|S|)^{O(d^3/\kappa^2)} \alpha^2$.  By the union bound, we thus have \eqref{rhogi} for all distinct $1 \leq i,j \leq D+1$ with probability at least $1 - O(|S|)^{O(d^3/\kappa^2)} \alpha^2$.  Meanwhile, for $i=j$ we have
$$
|\rho(g_i) \nabla(u)(g_j)| \geq 1/|S|$$
by hypothesis.  Applying \eqref{rhogi} we see that
$$
|\nabla(v_{g_i})(g_j)| \leq \frac{1}{200 |S| D}$$
for $i \neq j$ and
$$
|\nabla(v_{g_i})(g_j)| \geq \frac{1}{2 |S|}$$
for $i=j$.  The matrix $(\nabla(v_{g_i})(g_j))_{1 \leq i,j \leq D+1}$ is then diagonally dominant and thus invertible; however, the $v_{g_1},\ldots,v_{g_{D+1}}$ lie in a $D$-dimensional space and thus must have a linear dependence.  This leads to a contradiction unless the stated event occurs with zero probability; this forces $\alpha \gg O(|S|)^{-O(d^3/\kappa^2)}$, and the claim follows.
\end{proof}

We can combine this with Theorem \ref{trivdir} and Proposition \ref{infilip} to obtain a crucial reduction in growth order, from $d$ to approximately $d-1$:

\begin{corollary}[Reduction in growth order]\label{trivcor}  Let $0 < \kappa < 0.1$, $R_0, d \geq 1$, and let $(G,S)$ be a $(R_0,d)$-growth group.  Assume the bounds \eqref{kappa-small}, \eqref{r0-big} for some sufficiently large $C$.  Then at least one of the following holds:
\begin{itemize}
\item $G = B_S(R_0^{\exp(\exp(\kappa^{-O(1)}))})$.
\item There exists a $(R_0^{1-\kappa^{10}},d/(1-\kappa^{10}))$-growth subgroup $(G',S')$ of $G$ of index at most $R_0^{\kappa^{10}}$ and a positive integer $l = O(\kappa^{-6})$ such that $(G')^{(l+1)}$ is generated by a set $(S')^{(l+1)} \subset B_{S'}(R_0^{2\kappa})$ for which $( (G')^{(l+1)}, (S')^{(l+1)})$ is a $(R_0^{1-4\kappa}, (d-1+6\kappa)/(1-4\kappa))$-growth group.
\end{itemize}
\end{corollary}

\begin{proof}  Applying Proposition \ref{infilip} we see that either the first conclusion holds, or there exists a $R_0^{-K}$-harmonic Lispchitz function with $|\nabla u(\id)| \geq 1/|S|$, where $K$ is defined by \eqref{rbit}. Applying Theorem \ref{trivdir} we can ensure that \eqref{ughr-2} holds for this value of $u$; iterating this we see in particular that
\begin{equation}\label{ughr-3}
u(gx) = u(x) + O( R_0^{-40d})
\end{equation}
for all $e \in B_{(S')^{(l+1)}}(R_0^{1-4\kappa})$, $x \in B_S(R_0^{1-4\kappa})$.

By Proposition \ref{lorrange} (and cleaning up the constants) we have
$$ \sup_{x \in B_S(R_0^{1-4\kappa})} |u(x)-u(\id)| \geq 10 R_0^{1-6\kappa}$$
and thus there exists a path of length at most $R_0^{1-4\kappa}$ starting at the origin on which $u$ fluctuates by at least $\geq 10 R_0^{1-6\kappa}$.  Since $u$ also has Lipschitz constant at least $1$, we can thus find $M > R_0^{1-6\kappa}$ points $x_1,\ldots,x_M \in B_S(R_0^{1-4\kappa})$ such that $u(x_1),\ldots,u(x_M)$ all differ by at least $1$.  Combining this with \eqref{ughr-3} we see that the sets $B_{(S')^{(l+1)}}(R_0^{1-4\kappa}) \cdot x_m$ for $1 \leq m \leq M$ are disjoint.  But these sets all lie in $B_S(R_0)$, which has cardinality at most $R_0^d$.  We conclude that
$$ B_{(S')^{(l+1)}}(R_0^{1-4\kappa}) \leq R_0^d / M = R_0^{d-1+6\kappa}$$
and the claim follows.
\end{proof}

Finally, we can prove Proposition \ref{trivcor2}.

\begin{proof}[Proof of Proposition \ref{trivcor2}]  Applying Lemma \ref{gen-red} with $\kappa = \frac{c}{d}$ for some sufficiently small absolute constant $c>0$, one can find a $(R_0^{c/d}, R_0^{c/d})$-subgroup $(\tilde G,\tilde S)$ of $G$ which is a $(R_0^{1-\frac{1}{100d}}, d-0.01)$-growth group and with $|\tilde S| \ll O(1)^{d^2}$ and $G = \tilde G \cdot B_S( R_0^{\frac{1}{100d}})$.  We then apply Corollary \ref{trivcor} to $(\tilde G, \tilde S)$ with $\kappa := \frac{c}{d^2}$ for some sufficiently small absolute constant $c>0$.  If the first conclusion of Corollary \ref{trivcor} holds, then the properties relating $(G,S)$ to $(\tilde G,\tilde S)$ imply that $G = B_{S}(R_0^{\exp(\exp(O(d)^{O(1)}))})$.  If the second conclusion of Corollary \ref{trivcor} holds, then the claim follows after substituting in the value of $\kappa$ (and replacing $l$ by $l+1$), and using Lemma \ref{trans}.  (Observe from Lemma \ref{index} that $G'$ is a $(R_0^{\kappa^{10}},R_0^{\kappa^{10}},1)$-subgroup of $G$.)
\end{proof}

\section{First step for Proposition \ref{qmw}: reduction to a cyclic base}\label{wolfsec-1}

We now begin the proof of Proposition \ref{qmw}.  In this section we execute the first step of this proof, which is to reduce to the case when the base group $L$ is cyclic rather than solvable.

By replacing $S$ with $S \cup S_H$, replacing $R_0$ by $R_0^{9/10}$, and adjusting $d$ and $C$ appropriately, we may assume that $S_H \subset S$ in Proposition \ref{qmw}, thus $(H,S_H)$ is now a $(\infty,1)$-subgroup of $(G,S)$.

We now claim that Proposition \ref{qmw} follows from the $l=1$ case of this proposition, and specifically from

\begin{proposition}[Semi-quantitative Milnor-Wolf theorem, first reduction]\label{qmw1}  Let $r, R_0, d \geq 1$, and suppose that
$$ R_0 \geq C^d r^C$$
for some sufficiently large absolute constant $C$.  Suppose we have a short exact sequence
$$ 0 \to H \to G \to A$$ 
of groups, where $G = (G,S)$ is a $(R_0,d)$-growth group, $(H,S_H)$ is a virtually $r$-nilpotent $(\infty,1)$-subgroup of $(G,S)$, and $A$ is abelian.  Then $(G,S)$ is virtually $r+O(1)^d$-nilpotent.
\end{proposition}

To see how Proposition \ref{qmw1} implies Proposition \ref{qmw}, we induct on $l$.  The case $l=1$ already follows from Proposition \ref{qmw1}, so suppose that $l \geq 2$ and the claim has already been proven for $l-1$.  

Let $\kappa := \frac{1}{100 dl}$, then by Lemma \ref{zonk} we may find a set of generators $S' \subset B_S(R_0^\kappa) \cap [G,G]$ for $[G,G]$.  If we then set $\tilde G := \langle [G,G], H \rangle$ and $\tilde S := S_H \cup S' \subset B_S(R_0^\kappa)$, we see that $(\tilde G, \tilde S)$ is a $(R_0^{1-\kappa}, d/(1-\kappa))$-growth group, and we have the short exact sequence
$$ 0 \to H \to \tilde G \to [N,N] \to 0.$$
Of course, $[N,N]$ is solvable of derived length at most $l-1$, so by induction hypothesis $(\tilde G,\tilde S)$ is virtually $r+(l-1) O(1)^d$-nilpotent.  (Note that even after iterating this induction hypothesis $l$ times, the order of growth $d$ of $\tilde G$ does not increase significantly, thanks to the choice of $\kappa$, so for the purposes of computing quantitative bounds one can treat $d$ as constant throughout this iteration.)  On the other hand, $G/\tilde G$ is abelian.  If we then apply Proposition \ref{qmw1} to the short exact sequence
$$ 0 \to \tilde G \to G \to G/\tilde G \to 0$$
we obtain the claim (for $C$ large enough).

Now suppose we are in the situation of Proposition \ref{qmw1}.  At present, the abelian group $(A,S_A)$ could have many generators; but we can cut down the number of generators to a quantity depending on $d$.  Indeed, as $G = (G,S)$ is an $(R_0,d)$-growth group, we see that $(A, S_A)$ is also an $(R_0,d)$-growth group, where $S_A$ is the projection of $S$ to $A$.  Applying Lemma \ref{gen-red}, we may find a $(R_0^{0.9}, d/0.9)$-growth $(R_0^{0.1}, R_0^{0.1})$-subgroup $(A', S_{A'})$ of $(A,S)$ with $|S_{A'}| \ll O(1)^d$.  Let $G'$ be the preimage of $A'$ in $G$, and let $S' := B_S(R_0^{0.1}) \cap A'$, then $(G',S')$ is a $(R_0^{0.1},R_0^{0.1})$-subgroup of $(G,S)$.  We may thus replace $G$, $A$, $S$ by $G'$, $A'$, $S'$ in Proposition \ref{qmw1} (adjusting $d$ and $C$ slightly), thus allowing us to reduce to the case when $A$ is generated by at most $m = O(1)^d$ elements $a_1,\ldots,a_m$, and furthermore that we may assume that elements are contained in the projection of $S$ to $A$.

As $(H,S_H)$ is a $(\infty,1)$-subgroup of $(G,S)$, $S_H$ is contained in $S$.  We can then discard all elements of $S$ other than those in $S_H$ and those that are projecting to $A$, and thus assume that $S$ takes the form
$$ S = S_H \cup \{ e_1,\ldots,e_m,e_1^{-1},\ldots,e_m^{-1}\}$$
where $e_1,\ldots,e_m \in G$ projects to the generators $a_1,\ldots,a_m$ of $A$.

An easy induction on $m$ then allows us to reduce to the one-dimensional case $m=1$, and specifically from

\begin{proposition}[Semi-quantitative Milnor-Wolf theorem, second reduction]\label{qmw2}  Let $r, R_0, d \geq 1$, and suppose that
$$ R_0 \geq C^d r^C$$
for some sufficiently large absolute constant $C$.  Suppose we have a short exact sequence
$$ 0 \to H \to G \to A$$ 
of groups, where $(G,S)$ is a $(R_0,d)$-growth group, $(H,S_H)$ is a virtually $r$-nilpotent subgroup, and $A$ is cyclic.  Suppose also that $S = S_H \cup \{e,e^{-1}\}$ for some $e \in G$.  Then $G$ is virtually $r+1$-nilpotent.
\end{proposition}

It remains to establish Proposition \ref{qmw2}.  This will be accomplished in Section \ref{wolfsec-4}, after some preliminaries in Sections \ref{wolfsec-2}, \ref{wolfsec-3}.

\section{Second step for Proposition \ref{qmw}: slow growth of iterated conjugation}\label{wolfsec-2}

Suppose we are in the situation of Proposition \ref{qmw2}.  The short exact sequence forces $H$ to be a normal subgroup of $G$, and so the generator $e$ induces an automorphism $T: H \to H$ defined by the conjugation operation $T h := e h e^{-1}$.  As $S_H$ generates $H$, we see that there exists some radius $R$ such that $T(S_H), T^{-1}(S_H) \subset B_{S_H}(R)$.  This gives rise to the crude bound
$$ \| T^n h \|_{S_H} \leq R^{|n|} \|h\|_{S_H}$$
for all $n \in \Z$ and $h \in H$, where the norms $\| \|_{S_H}$ were defined in Definition \ref{rdg-def}.

This bound is useless for our purposes because the argument gives no effective bound on $R$.  However, it turns out that one can use the polynomial growth hypothesis $|B_S(R_0)| \leq R_0^d$ to obtain a much stronger bound, after adjusting the generating set $S$ slightly.  Namely, we have

\begin{proposition}[Slow growth]\label{slowg}  Let $R_0, d \geq 1$ be such that $R_0 \geq C^d$ for some sufficiently large absolute constant $C$.  Suppose we have a short exact sequence
$$ 0 \to H \to G \to A$$ 
of groups, where $A$ is cyclic.  Suppose also that $G = (G,S)$ is a $(R_0,d)$-growth group with $S = S_H \cup \{e,e^{-1}\}$, where $S_H \subset H$ generates $H$, and the projection $e$ of $S$ to $A$ generates $A$.  Then there exists a symmetric set $\tilde S \subset B_S(R_0^{1/10}) \cap H$ generating $H$ such that
$$
\| T^n h \|_{\tilde S} \ll \exp( |n| / R_0^{0.01} ) \| h \|_{\tilde S}
$$
for all $n \in \Z$ and $h \in H$, where $T: H \to H$ is the conjugation map $Th := ehe^{-1}$.
\end{proposition}

\begin{proof}  By hypothesis, $|B_S(R_0)| \leq R_0^d$.
From the pigeonhole principle and the lower bound on $R_0$, we may then find a radius $R_0^{1/30} \leq R \leq R_0^{1/20}$ such that
$$ |B_S(10 R) \cap H| \leq O(1)^d |B_S(R) \cap H|.$$
Fix this $R$.  For every $0 \leq N \leq R$, consider the sets
$$ A_N := \bigcup_{-N \leq n \leq N} T^n( B_S(R) \cap H ).$$
Observe that the $A_N$ are symmetric subsets of $H$ that are increasing in $N$, and that $B_S(R) \cap H \subset A_N \subset A_N \cdot (B_S(R) \cap H) \subset B_S(10R) \cap H$, and thus
$$ |B_S(R) \cap H| \leq |A_N \cdot (B_S(R) \cap H)| \leq O(1)^d |B_S(R) \cap H|.$$
By another application of the pigeonhole principle, there exists $0 \leq N \leq R/2$ such that
$$
|A_{N + C^{-d} R} \cdot (B_S(R) \cap H)| < |A_N \cdot (B_S(R) \cap H)| + |B_S(R) \cap H|
$$
for some large absolute constant $C$.

Fix this $R$, and set $\tilde S := A_N$, thus $\tilde S$ is a symmetric subset of $B_S(R_0^{1/10}) \cap H$ which contains $S_H$ and thus generates $H$.  For every $-C^{-d} R \leq n \leq C^{-d} R$, we see from construction that
\begin{equation}\label{contra}
 |(\tilde S \cup T^n \tilde S) \cdot (B_S(R) \cap H)| < |\tilde S \cdot (B_S(R) \cap H)| + |B_S(R) \cap H|.
 \end{equation}
This implies that
$$ \tilde S \cup T^n \tilde S \subset \tilde S \cdot (B_S(R) \cap H) \cdot (B_S(R) \cap H)^{-1},$$
for if there was an element $x \in \tilde S \cup T^n \tilde S$ which did not lie in $\tilde S \cdot (B_S(R) \cap H) \cdot (B_S(R) \cap H)^{-1}$, then the set $x \cdot B_S(R) \cap H)$ would lie in $(\tilde S \cup T^n \tilde S) \cdot (B_S(R) \cap H)$ but be disjoint from $\tilde S \cdot (B_S(R) \cap H)$, contradicting \eqref{contra}.  Since $B_S(R) \cap H$ is contained in the symmetric set $\tilde S$, we conclude that
$$ T^n \tilde S \subset B_{\tilde S}(3)$$
and thus (as $T^n: H \to H$ is an automorphism)
$$ \| T^n h \|_{\tilde S} \leq 3 \|h\|_{\tilde S}$$
for all $h \in H$ and $-C^{-d} R \leq n \leq C^{-d} R$.  From iteration we then obtain
$$ \| T^n h \|_{\tilde S} \ll \exp( C^d |n| / R ) \|h\|_{\tilde S}$$
for all $h \in H$ and $n \in \Z$, and the claim then follows from the lower bounds on $R$ and $R_0$.
\end{proof}

\section{Third step for Proposition \ref{qmw}: the case of an action on lattices}\label{wolfsec-3}

We are still preparing for the proof of Proposition \ref{qmw2} (and hence Proposition \ref{qmw}). Proposition \ref{slowg} places us in the setting of an automorphism $T: H \to H$ on some virtually nilpotent group $H$ whose iterates grow very slowly with respect to some word norm $\| \|_{S'}$.  In this section we study a key model case of this situation, in which $H = \Z^D = (\Z^D,+)$ is a lattice of bounded dimension, and the word norm $\| ||_{S'}$ is replaced by the Euclidean norm $|\cdot|$ on $\R^D$ (and hence $\Z^D$).  The study of this case is central to all arguments of Milnor-Wolf type, see e.g. \cite{wolf}, or the appendix by Tits in \cite{gromov}.  The main result to establish in this section is as follows.

\begin{proposition}[Dichotomy between periodicity and exponential growth]\label{dich}  Let $D \geq 1$, and let $T \in SL_D(\Z)$ be an invertible linear transformation $T: \Z^D \to \Z^D$.  Then at least one of the following statements hold.
\begin{itemize}
\item (Periodicity) There exists a non-zero vector $w \in \Z^D$ and an integer $1 \leq n \leq D^{O(1)}$ such that $T^n w = w$.
\item (Growth)  For any $N \geq 1$, there exists a non-zero vector $v = v_N \in \Z^D$ such that
\begin{equation}\label{tnv}
 |T^N v| \gg \exp( c N / D^{O(1)} ) |v|
\end{equation}
for some absolute constant $c>0$.
\end{itemize}
\end{proposition}

\begin{proof}  Let $\lambda_1,\ldots,\lambda_D \in \C$ be the eigenvalues of $T$ (counting multiplicity).  Then $\lambda_1 \ldots \lambda_D = \det(T) = \pm 1$, which implies that $\max_{1 \leq j \leq D} |\lambda_j| \geq 1$.  

Suppose first that $\max_{1 \leq j \leq D} |\lambda_j| = 1$, then all the $\lambda_j$ are algebraic integers whose Galois conjugates all lie on the unit circle.  By a classical result of Kronecker\cite{kron} (or Remark \ref{krem} below), this implies that the $\lambda_j$ are all roots of unity.  If one of the $\lambda_j$ is a primitive $n^{th}$ root of unity, then the degree $\phi(n)$ of that root cannot exceed $D$.  Elementary number theory (using the prime factorization of $n$) yields that $\phi(n) \geq c_\varepsilon n^{1-\varepsilon}$ for any $\varepsilon > 0$ and some constant $c_\varepsilon > 0$, and thus we have $n = O(D^{O(1)})$.  Then $T^n-I$ has non-trivial kernel, and we obtain the periodicity claim.

Now suppose instead that $\max_{1 \leq j \leq D} |\lambda_j| > 1$.  Applying a result of Dobrowolski \cite{dob}, we conclude in fact that $\max_{1 \leq j \leq D} |\lambda_j| > 1 + c D^{-O(1)}$ for some $c > 0$ (in fact the more precise bound $1 + c \frac{\log^3 D}{D \log \log^3 D}$ is known).  In particular we may assume that $|\lambda_1| \geq 1 + c D^{-O(1)}$.
Let $v_1$ be an eigenvector of $\lambda_1$, then clearly
$$ |T^N v_1| \gg \exp( O(1)^{-D} N ) |v_1|.$$
From the triangle inequality we see that either the real or complex part of $v_1$ obeys a similar growth bound.  Approximating this real or complex part by a non-zero vector with rational coefficients and then clearing denominators, we obtain the claim.
\end{proof}

\begin{remark}\label{krem} For our applications, one may replace Dobrowolski's lower bound of $1 + c D^{-O(1)}$ here with the more elementary bound of $1 + \exp( - O(D) )$, the proof of which we sketch as follows.  Suppose we had an algebraic integer $\lambda$ of degree $D$, all of whose Galois conjugates were at most $1+C^{-D}$ for some large absolute constant $C$, but which was not a root of unity.  Using the minimal polynomial of $\lambda$, one can find a diagonalisable transformation $T: \C^{D'} \to \C^{D'}$ for some $D' \leq D$ with integer coefficients and with eigenvalues equal to these Galois conjugates.  One can then use the corresponding eigenvectors to design a symmetric convex body $B$ with the property that $T^n(B) \subset 2 \cdot B$ for all $n \leq (C/10)^D$.  By rescaling, one may assume that $B$ contains a lattice vector $v$ in $\Z^D$ on its boundary, but no non-zero lattice vector in its interior.  As $\lambda$ is not a root of unity, the images of $T^n v$ are all distinct, and so $2 \cdot B$ contains at least $(C/10)^D$ lattice vectors; but standard volume packing arguments show that this cannot be the case for $C$ large enough.
\end{remark}

\section{Final step for Proposition \ref{qmw}: eliminating the finite factors}\label{wolfsec-4}

We are now ready to prove Proposition \ref{qmw2}.   Let $r, R_0, d, H, G, A, S, S_H, e$ be as in Proposition \ref{qmw2}.   

By hypothesis, $H$ contains a finite index nilpotent subgroup $H'$ of 
Hirsch length at most $r$, and thus step $s \le r$. As is well known, $H'$ 
then 
contains a finite index subgroup $H''$ which is torsion free.

The group $H''$ need not be $T$-invariant. However, observe from Legendre's theorem that $h^{|H:H''|} \in H''$ for all $h \in H$.  Thus, if we let $H'''$ be the group generated by $\{ h^{|H:H''|}: h \in H \}$, then $H'''$ is a $T$-invariant normal subgroup of $H''$.  The nilpotent quotient group $H''/H'''$ 
is generated by finitely many 
torsion elements, hence it is finite. 
We conclude that $H'''$ is a $T$-invariant finite index torsion free
 subgroup of $H$, and thus has Hirsch length $r$. Let $H'''_i$ be an upper
central series for $H'''$ terminating at $H'''_s := H'''$.
By torsion freeness, 
we can 
identify $H'''_i/H'''_{i-1}$ with $\Z^{d_i}$ for all $i \leq s$, 
with $\Sigma d_i = r$.

Meanwhile, by Proposition \ref{slowg}, there exists a set $\tilde S \subset B_S(R_0^{1/10}) \cap H$ generating $H$ such that
\begin{equation}\label{nosey}
\| T^n h \|_{\tilde S} \ll \exp( |n| / R_0^{0.01} ) \| h \|_{\tilde S}
\end{equation}
for all $n \in \Z$ and $h \in H$, where $T: H \to H$ is the conjugation map $Th := ehe^{-1}$.

We claim that the $\tilde S$ norm and $S'''$ norm are comparable on $H'''$, i.e. there exists an $M > 0$ such that
\begin{equation}\label{moist}
 M^{-1} \| h \|_{S'''} \leq \| h \|_{\tilde S} \leq M \|h\|_{S'''}
\end{equation}
for all $h \in H'''$.
(We do not claim an effective bound on $M$.)  The second inequality follows for sufficiently large $M$ since each element of $S'''$ is generated by $\tilde S$.  To get the former inequality, observe as in Remark \ref{joe} that we can partition $H$ into finitely many cosets $x_1 \cdot H''', \ldots, x_m \cdot H'''$ (with $x_1=\id$, say) with relations $e x_i = x_{j_{e,i}} g_{e,i}$ for all $1 \leq i \leq m$, $e \in \tilde S$ and some $1 \leq j_{e,i} \leq m$, $g_{e,i} \in H'''$; iterating these relations we see that any product of $R$ elements of $\tilde S$ can be expressed as the product of one of the $x_i$ times $R$ of the $g_{e,i}$; in particular, any product of $R$ elements of $\tilde S$ that lie in $H'''$ can be expressed as the product of $R$ of the $g_{e,i}$, giving the desired inequality for some $M$.

Inserting \eqref{moist} into \eqref{nosey} we conclude that
\begin{equation}\label{joy}
\| T^n h \|_{S'''} \ll M^2 \exp( |n| / R_0^{0.01} ) \| h \|_{S'''}
\end{equation}
for all $h \in H'''$ and $n \in \Z$.

We shall now argue similarly to Tits in his appendix to 
Gromov's~\cite{gromov}, but in a quantitative manner. 
The automorphism $T$ preserves each $H'''_i$ and thus acts on the successive 
abelian quotients which are torsion free as well (cf.~\cite{Mal1}): 
$H'''_i/H'''_{i-1} \cong \Z^{d_i}$.  Beginning from the top, 
by Proposition \ref{dich}, 
either $H'''_s/H'''_{s-1} \cong \Z^{d_s}$ contains a non-zero periodic vector 
with period at most $O(d_s)^{O(1)} = O(r)^{O(1)}$, or else one 
can find for any $n \geq 1$ a non-zero vector $v$ in $\Z^{d_s}$ such that 
$$
 |T^n v| \gg \exp( n / O(r)^{O(1)} ) |v|.
$$
But by the lower bound on $R_0$, the latter possibility contradicts \eqref{joy} if one takes $n$ sufficiently large, since one can use the norm in $\Z^{d_s}$ as a lower bound for the $S'''$ norm.  Thus $\Z^{d_s}$ contains a non-zero periodic vector $v$ with some period $1 \leq p_0 \ll r^{O(1)}$.  This implies that the operator $T^{p_0} - I$ (viewed as a linear transformation on $\Z^{d_s}$) 
maps $\Z^{d_s}$ to a $T$-invariant subgroup of $\Z^{d_s}$ of infinite 
index.  Hence the map $\Phi_{p_0}: h \mapsto (T^{p_0} h) h^{-1}$ 
maps $H'''_{(0)}:=H'''_s/H'''_{s-1}$ into a $T$-invariant 
subgroup of it, $H'''_{(1)}$ (of strictly smaller dimension).  We can 
iterate this process, finding a $1 \leq p_1 \leq O(r)^{O(1)}$ 
such that $\Phi_{p_1}$ maps $H'''_{(1)}$ into a $T$-invariant subgroup 
$H'''_{(2)}$ of strictly smaller dimension, until these groups vanish, 
i.e., $T$ takes $H'''_s=H'''$ into $H'''_{s-1}$.  But then we can work 
on the abelian quotient $H'''_{s-1}/H'''_{s-2} \cong \Z^{d_{s-1}}$ 
instead.  Standard nilpotent algebra then tells us that the \emph{square root} $||^{1/2}$ of the norm on $\Z^{d_{s-1}}$ is essentially a lower bound 
for the $S'''$ norm (continuing to rely heavily on torsion freeness), 
and we can continue the argument much as before, reducing the 
dimension of the $T$-invariant subgroups until they fall into $H'''_{s-2}$, 
then $H'''_{s-3}$, etc., eventually collapsing to the identity. 
Since $H'''$ had Hirsch length $r$, we can thus
find $1 \leq p_1,\ldots,p_m \ll r^{O(1)}$ for some $1 \leq m \leq r$ 
and $T$-invariant subgroups
$$ H''' = H'''_{(1)} \geq H'''_{(2)} \geq \ldots \geq H'''_{(m+1)} = \{\id\}$$
such that $\Phi_{p_i}$ maps $H'''_{(i)}$ to $H'''_{(i+1)}$ for all $1 \leq i \leq m$. 
 If we let $P := p_1 \ldots p_m \ll r^{O(r)}$ be the product of all these 
periods, it follows that $T^P$ acts unipotently on $H'''$.  This implies that 
$$ \{ e^{nP} h: n \in \Z, h \in H''' \}$$
is a finite-index subgroup of $G$ which is nilpotent of Hirsch length 
(and thus step) at most $r+1$, and Proposition \ref{qmw2} follows.

\section{Effectivization}\label{effective}

In this section we discuss some fully effective results including 
Theorem~\ref{polycyclic} and the case of torsion free groups, and 
mainly remark on the modifications needed in the above arguments to make the bound on the quantity $K(R_0,d)$ in Theorem \ref{main-thm} effective.  However, we will not provide complete details here for the latter, as they are rather lengthy, and the final bound on $K(R_0,d)$ obtained by this process is quite poor.

\subsection{Some fully quantitative results}\label{quantitative}

We begin by mentioning one quantitative geometric application of 
Theorem~\ref{main-thm}, based on Milnor's result~\cite{milnor1} as used in \cite[page 72]{gromov}.

\begin{corollary}[to Theorem~\ref{main-thm}]\label{fundamental}
Let $(V,d)$ be a complete Riemannian manifold of dimension $n$, and $K\ge 0$ 
be such that 
the values
of the Ricci tensor on the unit tangent bundle at all points is bounded from
below by $-(n-1)K$. Let $\Gamma$ be a group of isometries of $V$ generated by the finite subset $S$. For a point $v \in V$ define:
$$
\delta _v:=\inf \{d(\gamma v,v)|\gamma \in \Gamma\} \quad \quad
\Delta _v:=\sup \{d(\gamma v,v)|\gamma \in S \}
$$
If for some $v \in V$ and some $R> \exp(\exp(C(2n)^C)$ the inequality:
$$
4{\frac{\Delta_v}{\delta_v}} \exp(2\pi \Delta_v{\sqrt K}R) <R
$$ 
holds, where $C$ is the absolute constant of Theorem~\ref{main-thm}, 
then $\Gamma$ is virtually nilpotent.
\end{corollary}

This follows immediately from Milnor's inequality\cite{milnor1}
$$ |B_S(R)| \leq 4^n (\frac{\Delta_v}{\delta_v})^n R^n \exp( 2\pi \Delta_v \sqrt{K} R )$$
combined with the main Theorem~\ref{main-thm}, to which the condition in Corollary~\ref{fundamental} 
is tailored with the value $d=2n$. Of course, the Corollary is of interest in situations where $K$, or both $\Delta_v, \delta_v$ and their quotient, are small.

\smallskip
We next discuss Theorem~\ref{polycyclic}. This result is in fact a direct
consequence of Proposition~\ref{trivcor2} above. More precisely, this Proposition
immediately implies the following:

\begin{theorem}\label{polycyclic2} 
Let $d,R_0$ be as
in Theorem~\ref{main-thm}. There exists effective (explicit) functions
$A(R_0,d)$, $B(d)$ such that for every $(R_0,d)$-growth group $(G,S)$
a sequence of at most $B(d)$ operations of the type below
reduce the group to the trivial group:

\noindent {\bf 1.} Passing to a finite index subgroup of index bounded by
$A(R_0,d)$.

\noindent {\bf 2.} Passing to the kernel of a homomorphism to a cyclic
group (finite or infinite).

\end{theorem}

The proof of this result goes simply by inspecting Proposition~\ref{trivcor2}.
In each step where {\bf 1} occurs, the worst bound on $A(R_0,d)$ comes
from possibility (i) in Proposition~\ref{trivcor2}, where $A(R_0,d)$ has to be taken
as an upper bound for the size of the ball $B_S(R_0^{\exp(\exp(O(d)^{O(1)}))})$. Otherwise,
we pass to the finite index subgroup $G'$ appearing in (ii), for
which the inclusion $G \subseteq B_S(R_0)G'$ holds, giving a bound
$R_0^d$ on its index. By Proposition~\ref{trivcor2} the number of operations
$B(d)$ performed in this process is obviously $O(d^{O(1)})$, thus
completing the proof of Theorem~\ref{polycyclic2}.

To complete the proof of Theorem \ref{polycyclic2}, notice that every group which can be reduced to the trivial
group using the operations {\bf 1} and {\bf 2} above is virtually polycyclic in an effective manner, by applying the following lemma as many times as necessary:

\begin{lemma} If a group $G$ has a subgroup $G'$ 
of index at most $I$ which contains a polycyclic normal subgroup of index at most $J$,
then $G$ has a polycyclic normal subgroup of index at most $(I\cdot J)!$.
\end{lemma}

\begin{proof} By hypothesis, $G$ has a polycyclic subgroup $G''$ of index at most $IJ$.  The group $G$ acts by left multiplication on the quotient set $G/G''$; the stabilizer $H$ of this action then has index at most $(IJ)!$, and is a normal subgroup of $G$ that is contained in $G''$.  Since any subgroup of a polycyclic group is polycyclic, the claim follows.
\end{proof}

\begin{remark} Of course, only type {\bf 1} steps contribute to the computation
of the total index (this makes the computation effective; unfortunately,
the lack of reasonable control over the size of the {\it finite} quotients
appearing in {\bf 2} is responsible for our inability to get a similarly
effective result concerning the index of a {\it nilpotent} subgroup).
\end{remark}

Finally, the following observation
shows that torsion is the only obstacle at this point to obtaining a fully 
effective version of Gromov's Theorem, an issue which will be shortly discussed
in more detail over the next subsection.

\begin{corollary}[to Theorem~\ref{main-thm}]\label{torsion-free}
Retain the assumptions of Theorem~\ref{main-thm}, but assume further that
$G$ is torsion free. Then a finite index nilpotent subgroup as in the 
Theorem can be found with index $q$ satisfying
$$
\log \log \log q < C^d
$$
where $C$ is an absolute (effective) constant.

\end{corollary}

This follows immediately from Corollary~\ref{allscales} and a result of
Auslander and Schenkman -- see~\cite[page 71]{gromov} and the reference 
therein.

\subsection{Towards full effectivization of Theorem~\ref{main-thm}}
For the purpose of full quantification one must forego the 
compactness argument in 
Section \ref{compact-sec} that allowed us to reduce Theorem \ref{main-thm} to Theorem \ref{main-thm2}.  Instead, one must replace ``virtually $r$-nilpotent'' by something more like ``$(K,R,s,D)$-virtually nilpotent'' throughout the arguments, making sure to keep the bounds on $K, R$ effective.  

Of the three components of the proof in Section \ref{reduct-sec}, the first two (Lemma \ref{base} and Proposition \ref{trivcor2}) are already completely effective (and do not use the notion of virtual $r$-nilpotency).  The main issue, therefore, is to locate an effective version of Proposition \ref{qmw}.  This proposition is applied about $O(d)$ times during the induction on the growth order $d$ in the proof of Theorem \ref{main-thm}, so any deterioration in the $K, R$ constants occurring in that proposition will need to be iterated $O(d)$ times to obtain the final bound on $K(R_0,d)$.  The arguments in Section \ref{wolfsec-1}, in which Proposition \ref{qmw} is reduced to Proposition \ref{qmw2}, are also quite effective, in the sense that any effective version of the latter can be converted by a routine modification of the arguments in that section to a quantitative version of the former.  But as before, the deterioration of the bounds will worsen due to the induction on the step $l$ and the dimension $m$ appearing in the arguments in that section (more precisely, the bounds in Proposition \ref{qmw} will basically be a $O(1)^d$-fold iteration of the bounds in Proposition \ref{qmw2}).  

The key proposition in Section \ref{wolfsec-2}, namely Proposition \ref{slowg}, is already effective.  The key proposition in Section \ref{wolfsec-3}, namely Proposition \ref{dich}, can also be made effective without much difficulty, for instance by using the Kronecker approximation theorem to quantitatively approximate a vector by one with rational coefficients, and using Cramer's rule to bound all the vectors that are constructed from linear algebra (e.g. locating an integer vector in the null space of a matrix with integer coefficients).  To deal with various denominators in Cramer's rule, one does need to obtain a quantitative lower bound on the difference $|\lambda_i - \lambda_j|$ between two eigenvalues of a matrix $T$ with integer coefficients, but this can be accomplished by using Galois theory to observe that the product $\prod_{\lambda_i \neq \lambda_j} (\lambda_i - \lambda_j)$ of all these eigenvalue gaps is a non-zero rational integer, and in particular has magnitude at least $1$.  We omit the details.

The only remaining parts of the argument which need more careful attention are those in Section \ref{wolfsec-4}, as it is here that one truly begins to exploit such qualitative notions as Hirsch length and 
torsion freeness.

Recall that a group $G = (G,S)$ is virtually $r$-nilpotent if it contains a finite index nilpotent subgroup $G' = (G',S')$ which is torsion free of 
Hirsch length $r$ 
(and thus step at most $r$).  The notion of finite index can be made quantitative using the concept of a $(K,R)$-subgroup as defined in Definition \ref{qfi-def}.  Nilpotency of a given step is also a quantitative concept (being nilpotent of step $s$ is equivalent to the $s$-fold commutators of the generating set $S'$ vanishing).  However, torsion freeness cannot be verified using only finitely many group operations on the generators and so should not be considered as quantitative.  (For instance, the additive group $\Z/N\Z$ with generating set $\{-1,+1\}$ has torsion, in contrast to the torsion-free group $\Z$ with the same set of generators, but it is only possible to distinguish the two groups after performing at least $O(\log N)$ group operations on the generators, which is unbounded as $N \to \infty$.)  This is a genuine problem in the proof, 
as torsion freeness is used crucially in Section \ref{wolfsec-4} in 
order to bound $\|T^n h \|_{S'''}$ from below using the norm of a projection of $T^n h$ to a free abelian group $\Z^{d_i}$.

It is thus natural to seek finitary substitutes for the concept of being torsion-free.  Call a abelian group $G = (G,+)$ generated by a set $S = \{ \pm e_1,\ldots,\pm e_D \}$ of generators $e_1,\ldots,e_D$ \emph{$M$-torsion-free} for some $M \geq 1$ if the sums $\{ n_1 e_1 + \ldots + n_D e_D: |n_1| + \ldots + |n_D| \leq M\}$ that make up $B_S(M)$ are all distinct.  This property, in contrast to the qualitative property of being torsion-free, can be verified in finite time using a finite number of group operations on generators, and so we consider this a quantitative property.  
A key ingredient in the proof that any finitely generated nilpotent group has
a finite index torsion free subgroup is that same result for 
abelian groups.  We now give a quantitative version of this fact.  It is convenient to extend the asymptotic notation $X \ll Y$ by allowing the constants to depend on additional parameters, indicated by subscripts on the $\ll$ symbol; for instance, $X \ll_D Y$ means that $X \leq C(D) Y$ for some quantity $C(D)$ depending only on $D$.  (Of course, to get effective bounds at the end of the day, it is important to ensure that all implied constants such as $C(D)$ depend in an effective fashion on parameters such as $D$.)

\begin{lemma}[Quantitative location of an $M$-torsion-free group]\label{mtf}  Let $G = (G,S,+)$ be a finitely generated abelian group with $D$ generators $S  = \{ \pm e_1, \ldots, \pm e_D \}$, and let $F: \R^+ \to \R^+$ be an arbitrary function.  Then there exists an integer $1 \leq M \ll_{F,D} 1$ and a $(M,1)$-subgroup $(G',S')$ of $G$ with $S' = \{ \pm f_1, \ldots, \pm f_{D'} \}$ for some $1 \leq D' \leq D$, such that $(G',S')$ is $F(M)$-torsion-free.
\end{lemma}

\begin{proof} (Sketch)  We use the ``rank reduction argument''.  We begin by initializing $(G',S')$ equal to $(G,S)$ and initializing $M := 1$, then $(G',S')$ is already an $(M,1)$-subgroup of $(G,S)$.  If $(G',S')$ is $F(M)$-torsion-free then we are done.  If not, write $S' = \{\pm f_1,\ldots,\pm f_{D'}\}$ for some $D'$ (which is initially equal to $D$).  Because of the failure of $(G',S')$ to be $F(M)$-torsion-free, we must have a non-trivial dependence
$$ n_1 f_1 + \ldots + n_{D'} f_{D'} = 0$$
where $0 < |n_1|+\ldots+|n_{D'}| \leq 2F(M)$.  Without loss of generality we may take $n_{D'}$ to have the largest magnitude, and in particular be non-zero.  Then we see that the subgroup $(G'',S'')$ of $(G',S')$ generated by $S'' := \{ \pm f_1,\ldots, \pm f_{D'-1} \}$ has finite index (indeed, the index is at most $|n_{D'}|$); more quantitatively, one can show that $(G'',S'')$ is a $(O( D' F(M) ),1)$-subgroup of $(G',S')$, which by Lemma \ref{trans} implies that $(G'',S'')$ is a $(M', 1)$-subgroup of $(G,S)$ for some $M'= O(D M F(M))^{O(1)}$.  One then replaces $(G',S'), D', M$ by $(G'',S''), D'', M'$ respectively, and iterates this procedure.  Since the rank $D'$ starts at $D$ and decreases by one at each stage, this algorithm terminates in at most $D$ steps, and the claim follows.
\end{proof}

\begin{remark} The eventual bound on $M$ is essentially a $D$-fold iteration of $F$.  Unfortunately, in applications $F$ has to be quite a rapidly growing function (exponential or worse), which leads to quite poor bounds, especially after the induction loops on the growth order $d$, the solvability index $l$, and the dimension $D$ that appear in the proof are carried out.  It is thus of interest to reduce the dependence on the torsion-free property (and also to reduce the length of the induction loops) in order to improve the bounds.
\end{remark}

In a similar vein to the above discussion, it is possible to define a quantitative notion of an ``$M$-torsion free nilpotent group'', and establish a quantitative version of the fact that every finitely generated nilpotent group contains a finite index torsion free subgroup, which roughly speaking 
asserts that for any given $F$, any $(s,D)$-nilpotent subgroup 
will have an $(M,M)$-subgroup which is $F(M)$-torsion free 
for some $M \ll_{F,s,D} 1$.  This will turn out to be a usable 
quantitative substitute for the qualitative notion of torsion freeness
 if $F$ is 
chosen to be sufficiently rapidly growing.  Indeed, the various 
vectors in $\Z^D$ which appear in the arguments in Section \ref{wolfsec-4} will have norm bounded by some effective function of the parameter $M$ (thanks to quantitative versions results such as Proposition \ref{dich}), and as long as $F(M)$ is much larger than the norm of these vectors, then the $M$-torsion free 
nilpotent groups involved will behave ``as if'' they are genuinely torsion free
 for the purposes of the computations being performed.  For instance, it would suffice to take $F(M) := \exp(\exp(\exp( (CdlM)^C )$ for some large absolute constant $C$.  If one does this, then the final bounds obtained on $K(R_0,d)$ are essentially an Ackermann function of $R_0$ and $O(1)^d$.

\end{document}